\documentclass[12pt,oneside,reqno]{amsart}
\usepackage[margin=1.5in]{geometry}
\usepackage{xcolor}
\usepackage{hyperref}
\usepackage{verbatim}

\newtheorem{theorem}{Theorem}[section]
\newtheorem{lemma}{Lemma}[section]
\newtheorem{corollary}{Corollary}[section]
\newtheorem{proposition}{Proposition}[section]

\newtheorem{remark}{Remark}[section]

\usepackage{bm}
\usepackage{graphicx}
\usepackage{subfigure}
\usepackage{color}

\newcommand{\bx}{\mbox{\boldmath $x$}}
\newcommand{\bv}{\mbox{\boldmath $v$}}
\newcommand{\bu}{\mbox{\boldmath $u$}}

\newcommand{\e}{\epsilon}

\newcommand{\Be}{\begin{equation}}
\newcommand{\Ee}{\end{equation}}
\def\bega{\begin{aligned}}
\def\enda{\end{aligned}}
\def\lw{\left}
\def\rw{\right}

\def\al{\alpha}
\def\e{\varepsilon}

\numberwithin{equation}{section}

\begin{document}
	
\title[Comparison between theories of gas mixtures]{On the comparison between phenomenological and kinetic theories of gas mixtures with applications to flocking}

\author{Gi-Chan Bae}
\address{Research institute of Mathematics, Seoul National University, Seoul 08826, Republic of Korea}
\email{gcbae02@snu.ac.kr}

\author{Seung-Yeal Ha}
\address{Department of Mathematical Sciences and Research Institute of Mathematics,
	Seoul National University, Seoul, 08826, Republic of Korea}
\email{syha@snu.ac.kr}

\author{Gyuyoung Hwang}
\address{Department of Mathematical Sciences, Seoul National University, Seoul 08826, Republic of Korea}
 \email{hgy0407@snu.ac.kr}

\author{Tommaso Ruggeri}
\address{Department of Mathematics and Alma Mater Research Center on Applied Mathematics AM$^2$,
	University of Bologna, Italy}
\email{tommaso.ruggeri@unibo.it}

\keywords{Conservation laws, gas mixture, flocking, kinetic theory, phenomenological theory}

\begin{abstract}
We study the compression between the phenomenological and kinetic models for a mixture of gases from the viewpoint of collective dynamics. In the case in which constituents are Eulerian gases, balance equations for mass, momentum, and energy are the same in the main differential part, but production terms due to the interchanges between constituents are different. They coincide only when the thermal and mechanical diffusion are sufficiently small. In this paper, we first verify that both models satisfy the universal requirements of conservation laws of total mass, momentum, and energy, Galilean invariance and entropy principle. Following the work of Ha and Ruggeri (ARMA 2017),  we consider spatially homogeneous models which correspond to the generalizations of the Cucker Smale model with the thermal effect. In these circumstances, we provide analytical results for the comparison between two resulting models and also present several numerical simulations to complement analytical results. 
	\end{abstract}
	
	\maketitle
	
	\tableofcontents
	
 \section{Introduction} \label{sec:1}
 \setcounter{equation}{0}
The mathematical theory of gas mixtures provides both challenging and stimulating problems for researchers in nonlinear sciences. Successful models can be deduced from either the continuum theory of fluids or the kinetic theory in the case of rarefied gases. In both cases, suitable equations can be derived to explain irreversible phenomena such as diffusion, heat transfer, and chemical reactions, etc. We refer to books \cite{Bose, Hutter,Rajagopal,Wilmansky} for the state-of-the-art results.  In particular, for the link between the macroscopic and mesoscopic approaches, we refer to the recent book \cite{BookNew} by Ruggeri and Sugiyama on Rational Extended Thermodynamics (RET).

The Truesdell theory \cite{Truesdell} for homogeneous mixtures within the framework of rational thermodynamics assumes that each component obeys the same balance laws as a single fluid, but there are production terms responsible for the interchange of mass, momentum, and energy between the components. The production terms are determined by universal principles such as Galilean invariance, the entropy principle, and the requirement that the whole mixture is conservative. In the case of a mixture where the individual constituents are Eulerian gases, the phenomenological theory of multi-temperatures yields a hyperbolic symmetric system of balance laws \cite{Simic1}. Recent studies, particularly those concerning shock waves, have focused on the investigation of the differential system \cite{Shighero} (see the book \cite{BookNew} or review papers \cite{Atkin, RuggeriWilma, RuggeriTAM, SimicParma}).

On the other hand, for rarefied gases, vast literature exists on the modeling of mixtures using the variants of the Boltzmann equation for both monatomic and polyatomic gases (see the classical book of Cercignani \cite{cercignani}). In particular, several BGK-type models were proposed for monatomic and polyatomic gases. We refer to the recent review \cite{Pirner} and reference therein. 
 
When individual constituents are Eulerian gases without heat conduction and viscosity, the macroscopic theory equipped with multi-temperatures and the first five moments associated with Boltzmann equations for mixtures yields the same principal part of the differential system, although production terms are different. For the direct comparison between two particle (or microscopic) models based on phenomenological theory and kinetic theory, we need to employ some normalization procedure (see Section \ref{sec:4.1}). To set up the stage, we begin with a brief description of two-particle models for flocking which can be regarded as a generalization of the Cucker-Smale model \cite{CS} for flocking.

We consider a group of $n$ Cucker-Smale particles with internal observables denoted by temperature. Let $\mathbf{x}_\alpha$, $\mathbf{u}_\alpha$, and $T_\alpha$ represent the position, diffusion velocity, and temperature of the $\alpha$-th particle, respectively. For the observable $(\mathbf{x}_\alpha, \mathbf{u}_\alpha, T_\alpha)$, we define the associated energy observable $E_\alpha$ as follows (assuming the internal energy is a linear function of temperature, specific heat, and constant density, equal to one for each constituent):
\[ E_\alpha := \frac{1}{2} |\bu_\alpha|^2 + T_\alpha, \quad \alpha \in [n] := \{1, \ldots, n \}. \]
Then, the phenomenological theory based Cucker-Smale (PB-CS) model reads as follows. 
\begin{equation} \label{A-2}
\begin{cases}
\displaystyle \frac{d\bx_{\al}}{dt} = \bu_{\al}, \quad \alpha \in [n], \\
\displaystyle \frac{d\bu_{\al}}{dt} = \frac{1}{n}\sum_{\beta=1}^n \phi_{\alpha \beta} \left(\frac{\bu_{\beta}}{T_\beta}-\frac{\bu_{\al}}{T_\al}\right), \\
\displaystyle \frac{d E_\alpha}{dt}  = \frac{1}{n}\sum_{\beta=1}^n \zeta_{\alpha \beta}\left(\frac{1}{T_\al} -\frac{1}{T_\beta} \right).
\end{cases}
\end{equation}
On the other hand, the kinetic theory (in the case of homogeneous solutions) based model for the Cucker-Smale (KB-CS) model also reads as follows.
\begin{equation}  \label{A-1}
\begin{cases}
\displaystyle \frac{d\bx_{\al}}{dt} = \bu_{\al}, \quad \alpha \in [n], \\
\displaystyle \frac{d\bu_{\al}}{dt} = \frac{1}{n}\sum_{\beta=1}^n a_{\alpha \beta}\left(\bu_{\beta}-\bu_{\al}\right),  \\
\displaystyle \frac{d E_\alpha}{dt} = \frac{1}{n}\sum_{\beta=1}^n a_{\alpha \beta}  (E_\beta - E_\alpha). 
\end{cases}
\end{equation}
Note that the PB-CS model \eqref{A-2} has been extensively studied from various angles in literature, e.g., emergent dynamics \cite{H-K-R, Ha_Ruggeri}, its hydrodynamic description \cite{HKMRZ2, KHKS}, mean-field limit \cite{HKMRZ1} and continuum limit \cite{TK} etc.  Thus, we mainly focus on the KB-CS model \eqref{A-1}. In contrast, as far as the authors know, the KB-CS model has not been considered in the literature. In this paper, we are interested in the following set of questions:
\begin{itemize}
\item
(Q1):~Can the proposed KB-CS model \eqref{A-1} exhibit collective dynamics?
\item
(Q2):~If so, what will be the quantitative similarities and differences between two models  \eqref{A-1} and \eqref{A-2}?
\end{itemize}
The main purpose of this paper is to address the above two questions by providing several qualitative and quantitative estimates. More precisely, our main results can be briefly summarized as follows. 

First, we identify different expressions for production terms in the phenomenological model and kinetic theory model for gas mixtures (see Section \ref{sec:2.2} and Section \ref{sec:3.1}). 

Second, we demonstrate that both models satisfy universal principles. In particular, we show that the system derived from the kinetic theory satisfies the entropy principle. This is not evident at first glance, since although the original kinetic model satisfies an $H$-theorem,  its reduced model using only a finite number of moments may not satisfy an entropy principle. This is typical in Rational Extended Thermodynamics (see \cite{BookNew})  in which we need to verify whether the reduced model satisfies an entropy principle or not  (see Theorem \ref{T3.1}).  

Third, we show that system \eqref{A-1} exhibits asymptotic flocking dynamics for all initial data (see Theorem \ref{T5.1}). This is an apparent difference from the PB-CS model in which flocking behaviors can be observed for well-prepared initial data (see Theorem\ref{T5.2}). 

Finally, we compare the similarity and discrepancy between the two particle models. For this comparison, we begin with an important observation by Ha and Ruggeri \cite{Ha_Ruggeri}, who noticed that in the isothermal case, homogeneous solutions of the phenomenological theory of mixtures coincide with that of the Cucker-Smale model for flocking \cite{CS}. Utilizing this analogy, they proposed a thermo-mechanical counterpart  (which is called the TCS model) of the Cucker-Smale model that contains energy equations, when the system has different internal energy of each constituent (temperatures). This provides a boost to a series of studies in both classical and
relativistic frameworks \cite{C-H-J-K2, C-H-J-K1, HLKR, H-K-R,HKMRZ1,HKMRZ2,TK, KHKS}. In Section \ref{sec:5}, we study the TCS model when production terms are derived from the kinetic theory and compare them with previous results obtained from the macroscopic phenomenological theory. This comparison is also interesting at the level of mixtures for homogeneous solutions. \newline

The rest of this paper is organized as follows. In Section \ref{sec:2}, we recall three universal principles, production terms, and a phenomenological theory-based Cucker-Smale model. In Section \ref{sec:3}, we present the parallel description of a kinetic theory-based approach for mixtures, and we present production terms and a heuristic derivation of the kinetic theory-based Cucker-Smale model. In Section \ref{sec:4}, we provide two normalized particle models which can be derived using two approaches such as the phenomenological theory and kinetic theory. In Section \ref{sec:5}, we study asymptotic equivalence of two-particle models in Section \ref{sec:4} in near equilibrium regime, and we also discuss discrepancy for the proposed models. In Section \ref{sec:6}, we provide several numerical examples and compare them with analytical results in previous sections. Finally, Section \ref{sec:7} is devoted to a summary and some remaining issues for future work. 
%%%%%%%%%%%%%%%%%%%%%%%%%%%
%
%. Section 2
%
%
%%%%%%%%%%%%%%%%%%%%%%%%%%%
\section{A phenomenological theory of Mixtures}  \label{sec:2}
\setcounter{equation}{0}
In this section, we study a phenomenological theory for gas mixture following the presentation in \cite{Ha_Ruggeri}.  In the context of rational thermodynamics,  the dynamic description of a mixture of $n$ component gases is based on the postulate that each component obeys the same balance laws that a single fluid obeys \cite{RET,BookNew, Truesdell}.  \newline
	
	 Let $\rho _{\alpha }, \bv_\alpha$  and $\varepsilon_\alpha$ denote the mass density,  the velocity and the specific internal energy of component $\alpha$, respectively and the quantities $\mathbf{q}_{\alpha}$ and $\mathbf{t}_{\alpha }$ are the heat flux, and the stress tensor. Then, governing balance laws are given by the differential relations for local densities of masses, momenta, and energies:
	\begin{equation}
	\begin{cases}  \label{RT_model}
\vspace{0.2cm}
	& \displaystyle \frac{\partial \rho _{\alpha }}{\partial t}+\mathrm{div}
	\,(\rho _{\alpha }\bv_{\alpha })=\tau _{\alpha }, \qquad    \alpha \in [n] := \{ 1,2,\ldots ,n \}, \\
\vspace{0.2cm}
	& \displaystyle \frac{\partial (\rho _{\alpha }\bv_{\alpha })}{%
		\partial t}+\mathrm{div}\,(\rho _{\alpha }\bv_{\alpha }\otimes
	\bv_{\alpha }-\mathbf{t}_{\alpha })=\mathbf{M}_{\alpha
	}, \\
	& \displaystyle \frac{\partial \left( \frac{1}{2}\rho _{\alpha }v_{\alpha
		}^{2}+\rho _{\alpha }\varepsilon _{\alpha }\right) }{\partial t}  \\
	& \displaystyle \hspace{1cm}  +~ \mathrm{div}
	\,\left\{ \left( \frac{1}{2}\rho _{\alpha }v_{\alpha }^{2}+\rho _{\alpha
	}\varepsilon _{\alpha }\right) \bv_{\alpha }-\mathbf{t}_{\alpha }%
	\bv_{\alpha }+\mathbf{q}_{\alpha }\right\} =e_{\alpha }, 
	\end{cases}
	\end{equation}%
	where $v_\alpha = |\bv_\alpha|$, $\tau _{\alpha }$, $\mathbf{M}_{\alpha }$, and $%
	e_{\alpha }$ represent production terms due to the interchanges 
	between components of mass, momentum, and energy.  Note that the stress tensor $\mathbf{t}_{\alpha }$ can be decomposed into the pressure part $-p_{\alpha }\mathbf{I}$
	and the viscous part ${\bm \sigma}_{\alpha }$:
	\begin{equation} \label{New-0}
	\mathbf{t}_{\alpha }=-p_{\alpha }\mathbf{I}+{\bm {\sigma}}_{\alpha}.
	\end{equation}
System \eqref{RT_model} corresponds to the particular case to the general system of balance laws:
	\begin{equation} \label{New-1}
\frac{\partial {\bf F}}{\partial t} + \frac{\partial {\bf F}^i}{\partial x_i} = {\bf f},
	\end{equation}
	where ${\bf F}, {\bf F}^i$ and ${\bf f}$ denote the local density, flux, and production terms, respectively. 
	\subsection{Universal principles} \label{sec:2.1}
	In this subsection, we present three universal principles to be used in the identification of production terms in the sequel. \newline
	\begin{enumerate}
	\item 
	$({\mathcal P}1)$ (Global conservation laws):~We assume that the total sum of production terms is zero:
	\begin{equation} \label{New-1-1}
		\sum_{\alpha =1}^{n}\tau _{\alpha }=0,\quad \sum_{\alpha =1}^{n}\mathbf{M}
		_{\alpha }=\mathbf{0},\quad \sum_{\alpha =1}^{n}e_{\alpha }=0.
	\end{equation}
	\vspace{0.2cm}
		\item 
		$({\mathcal P}2)$ (Galilean invariance):~For thermo-mechanical observables $\{ (\rho_\alpha, \bv_\alpha, e_\alpha) \}$,  we define the center of mass, bulk velocity and diffusion velocities as 
		\[
		\rho := \sum_{\alpha=1}^{n} \rho_\alpha, \quad \bv := \frac{1}{\rho} \sum_{\alpha =1}^{n} \rho_{\alpha}  \bv_\alpha, \quad \bu_\alpha :=\bv_\alpha - \bv.
		\]
		Then the Galilean invariance require that:
	       \begin{equation} \label{ProdTransf}
	       \tau_{\alpha} =\hat{\tau}_{\alpha}, \quad {\bf M}_\alpha =\hat{\tau}_{\alpha}\bv +\hat{{\bf M}}_{\alpha}, \quad e_{\alpha} =\hat{\tau}_{\alpha}\frac{|\bv|^{2}}{2}+\hat{{\bf M}}_{\alpha} \cdot \bv +\hat{e}_\alpha,
		\end{equation}
		where the hat quantities $(\hat{\tau}_{\alpha}, \hat{{\bf M}}_{\alpha}, \hat{e})$ are independent of bulk velocity $\bv$, and they depend on objective variables and on $\bu_\alpha$ that are frame independent diffusion velocities. For the proof see \cite{Simic1,BookNew}.

\vspace{0.2cm}

\item 
$({\mathcal P}3)$ (Entropy Principle):~We assume that for each $\alpha \in [n]$, there exists an entropy density $S_\alpha$ and supplementary scalar differential law:
\[
\frac{\partial \rho_\alpha  S_\alpha}{\partial t}+ \frac{\partial {\bm \Phi}_\alpha}{\partial {\bf x}} = \Sigma_\alpha,
\]
such that any solutions to system \eqref{RT_model} have a global entropy law with non-negative production:
\begin{equation}\label{EP}
\frac{\partial \rho   S}{\partial t}+ \frac{\partial {\bm \Phi}}{\partial {\bf x}} = \Sigma \geq 0,
\end{equation}
where  $\rho S, {\bf \Phi}$ and $\Sigma$ are given by the following relations:
\begin{equation}\label{GlobalS}
\rho S := \sum_{\alpha=1}^{n} \rho_\alpha S_\alpha, \qquad {\bm \Phi} :=\sum_{\alpha=1}^{n} {\bm \Phi}_\alpha, \qquad \Sigma :=\sum_{\alpha=1}^{n}  \Sigma_\alpha.
\end{equation}
\end{enumerate}

\vspace{0.5cm}

In the following subsections, we identify, in the case that each constituent is the Eulerian gas, the production terms based on three principles $({\mathcal P}1)$,  $({\mathcal P}2)$ and  $({\mathcal P}3)$.
\subsection{Production terms} \label{sec:2.2}
Consider a mixture in which each constituent is an Euler fluid  with
\[  {\bf q}_\alpha =0, \quad  {\bm \sigma}_\alpha=0,  \quad \forall~~\alpha \in [n]. \]
In this case, system \eqref{RT_model} becomes 
	\begin{equation} \label{Eulero}
	\begin{cases}
\vspace{0.2cm}
	& \displaystyle \frac{\partial \rho _{\alpha }}{\partial t}+\mathrm{div}
	\,(\rho _{\alpha }\bv_{\alpha })=\tau _{\alpha },  \quad \alpha \in [n],\\
\vspace{0.2cm}
	& \displaystyle\frac{\partial (\rho _{\alpha }\bv_{\alpha })}{
		\partial t}+\mathrm{div}\,(\rho _{\alpha }\bv_{\alpha }\otimes
	\bv_{\alpha }+p_\alpha {\bf I})=\mathbf{M}_{\alpha}, \\
	& \displaystyle\frac{\partial \left( \frac{1}{2}\rho _{\alpha }v_{\alpha
		}^{2}+\rho _{\alpha }\varepsilon _{\alpha }\right) }{\partial t}  +\mathrm{div}%
	\,\left\{ \left( \frac{1}{2}\rho _{\alpha }v_{\alpha }^{2}+\rho _{\alpha
	}\varepsilon _{\alpha } + p_\alpha\right) \bv_{\alpha } \right\} =e_{\alpha }, 
	\end{cases}
	\end{equation}
supplemented with thermal and caloric constitutive relations of state:
\begin{equation}\label{constitutive}
    p_\alpha \equiv p_\alpha(\rho_\al,T_\al), \quad \varepsilon_\al \equiv \varepsilon_\al (\rho_\al,T_\al).
\end{equation}
For system  \eqref{Eulero} - \eqref{constitutive},  entropy principle requires the existence of a main field $\mathbf{u} ^\prime \equiv \left(\Lambda^{\rho_{\alpha}},\boldsymbol{\Lambda}^{\bv_{\alpha}},\Lambda^{\varepsilon_{\alpha}} \right)$ such that
\begin{equation*}
    d(\rho_\alpha S_\alpha)=  \mathbf{u} ^\prime \cdot d{\bf F}={\Lambda}^{\rho_{\alpha}%
}d\rho_{\alpha}+{\boldsymbol{\Lambda}}^{\bv_{\alpha}}%
d(\rho_{\alpha}\bv_{\alpha})+{\Lambda}^{\varepsilon_{\alpha}%
}d\left(  \frac{1}{2}\rho_{\alpha}v_{\alpha}^{2}+\rho_{\alpha}\varepsilon
_{\alpha}\right),
\end{equation*}
where ${\bf F} = (\rho_\alpha, \rho_{\alpha}\bv_{\alpha},  \frac{1}{2}\rho_{\alpha}v_{\alpha}^{2}+\rho_{\alpha}\varepsilon
_{\alpha}).$ We refer to \cite{RS,BookNew} for details.   \newline

Note that the above relation and \eqref{GlobalS} imply
\begin{equation}
d(\rho S)=  
\sum_{\alpha=1}^{n}\left\{  {\Lambda}^{\rho_{\alpha}%
}d\rho_{\alpha}+{\boldsymbol{\Lambda}}^{\bv_{\alpha}}%
d(\rho_{\alpha}\bv_{\alpha})+{\Lambda}^{\varepsilon_{\alpha}%
}d\left(  \frac{1}{2}\rho_{\alpha}v_{\alpha}^{2}+\rho_{\alpha}\varepsilon
_{\alpha}\right)  \right\}. \label{MF_entr1}%
\end{equation}
In fact, the main field components were already computed in \cite{Simic1}:
\begin{equation} \label{MF_model1}
{\Lambda}^{\rho_{\alpha}}=\frac{-g_{\alpha}+\frac{1}{2}v_{\alpha}^{2}%
}{T_{\alpha}}, \quad{\boldsymbol{\Lambda}}^{\bv_{\alpha}}%
=-\frac{\bv_{\alpha}}{T_{\alpha}}, \quad{\Lambda}^{\varepsilon_{\alpha}}=\frac{1}{T_{\alpha}},
\end{equation}
where
\[
g_\alpha = \varepsilon_{\alpha } +\frac{p_\alpha}{\rho_\alpha} - T_\alpha S_\alpha, \quad  \text{is the chemical potential.}
\]
Moreover, we use the results in  \cite{Galileo,Simic1,BookNew} and \eqref{New-1} to see
\begin{equation} \label{New-2}
\Sigma_\alpha=\mathbf{u} ^\prime \cdot \mathbf{f}= \hat{\mathbf{u} }^{\prime}\cdot\hat{\mathbf{f}}=\hat{\Sigma}_\alpha =\hat{\Lambda}^{\rho_{\alpha}}\hat{\tau}_{\alpha}+\hat{\boldsymbol{\Lambda}}%
^{\bv_{\alpha}}\cdot\hat{\mathbf{M}}_{\alpha}+\hat{\Lambda}^{\varepsilon_{\alpha}}
\hat{e}_{\alpha}.
\end{equation}
Now, we use $\eqref{GlobalS}_3$, \eqref{New-2} and \eqref{EP} to find 
\begin{align}
\begin{aligned} \label{New-3}
\Sigma &=  \sum_{\beta=1}^{n}\left(
\hat{\Lambda}^{\rho_{\beta}}\hat{\tau}_{\beta}+\hat{\boldsymbol{\Lambda}}%
^{\bv_{\beta}}\cdot\hat{\mathbf{M}}_{\beta}+\hat{\Lambda}^{\varepsilon_{\beta}}%
\hat{e}_{\beta}\right) \\
&= \sum_{\beta=1}^{n}\left(
\frac{-g_{\beta}+\frac{1}{2}u_{\beta}^{2}%
}{T_{\beta}}\hat{\tau}_{\beta}- 
 \frac{{\bu _{\beta}}}{T_\beta}\cdot\hat{\mathbf{M}}_{\beta}+\frac{1}{T_\beta}%
\hat{e}_{\beta}\right) \geq 0.
\end{aligned}
\end{align}
On the other hand, it follows from \eqref{New-1-1} and \eqref{ProdTransf} that 
\begin{equation}  \label{New-4}
{\hat \tau}_n = -\sum_{\beta =1}^{n-1} {\hat \tau} _{\beta },\quad  \mathbf{{\hat M}}_{n} = -\sum_{\alpha =1}^{n-1}\mathbf{{\hat M}}_{\alpha },\quad    {\hat e}_{n} = - \sum_{\alpha =1}^{n-1} {\hat e}_{\alpha}.
\end{equation}
Finally, we use \eqref{New-3} and \eqref{New-4} to obtain
\begin{align}
\begin{aligned} \label{ResIneqNmix}
\Sigma &=  \sum_{\beta=1}^{n-1}\left(
\frac{-g_{\beta}+\frac{1}{2}u_{\beta}^{2}}{T_{\beta}}\hat{\tau}_{\beta}-\frac{{\bu_{\beta}}}{T_\beta}\cdot\hat{\mathbf{M}}_{\beta}+\frac{1}{T_\beta}\hat{e}_{\beta}\right) \cr 
&\qquad + 
\frac{-g_{n}+\frac{1}{2}u_{n}^{2}}{T_{n}}\hat{\tau}_{n}- \frac{{\bu_{n}}}{T_n}\cdot\hat{\mathbf{M}}_{n}+\frac{1}{T_n}\hat{e}_{n}  \\
&= \sum_{b=1}^{n-1}\bigg\{\left( - \frac{g_b-\frac{1}{2}u_{b}^{2}}{T_{b}}+\frac{g_n-\frac{1}{2}u_{n}^{2}}{T_{n}}\right)  \hat{\tau}_{b} \cr 
&\qquad +\left(\frac{\bu _{n}}{T_{n}}- \frac{\bu _{b}}{T_{b}}\right) \cdot
\hat{\mathbf{M}}_{b}+\left(\frac{1}{T_{b}}-\frac{1}{T_{n}}\right)  \hat{e}_{b}\bigg\}\geq0. 
\end{aligned}
\end{align}
%\begin{equation}
%	\Sigma=\bu ^\prime \cdot \mathbf{f}=\hat{\bu }^{\prime}\cdot\hat{\mathbf{f}}=\sum_{b=1}^{n-1}\left(
%	\hat{\Lambda}^{\rho_{b}}\hat{\tau}_{b}+\hat{\boldsymbol{\Lambda}}%
%	^{\bv_{b}}\cdot\hat{\mathbf{m}}_{b}+\hat{\Lambda}^{\varepsilon_{b}}%
%	\hat{e}_{b}\right)  \geq0,
%%\rho S =	\sum_{\alpha =1}^{n} \rho_\alpha  S_\alpha, \qquad 
%%{\bm \Phi} = \sum_{\alpha =1}^{n} {\bm \Phi}_\alpha, \qquad
%%\Sigma= \sum_{\alpha =1}^{n}\Sigma_\alpha.
%\end{equation}
	
%	In the case of an inert mixture of Eulerian gas for which ${\bm \sigma}_{\alpha } = {\bf q}_\alpha = \tau_\alpha =0$
% 	\begin{equation}
%	\Sigma=\sum_{b=1}^{n-1}  \left(
%	\frac{\bu _{n}}{T_{n}}- \frac{\bu _{b}}{T_{b}}\right)  \cdot
%	\hat{\mathbf{m}}_{b}+\left(  \frac{1}{T_{b}}-\frac{1}{T_{n}}\right)  \hat
%	{e}_{b}\geq0, \label{ResIneqNmix}%
%	\end{equation}
%
%		\begin{equation}
%		\Sigma=\sum_{\alpha=1}^{n}    \left( - \frac{g_\alpha-\frac{1}{2}u_{\alpha}^{2}}{T_{\alpha}%
%		}+\frac{g_\beta-\frac{1}{2}u_{\beta}^{2}}{T_{n}}\right)  \hat{\tau}_{\alpha}+\left(
%		\frac{\bu _{\beta}}{T_{\beta}}- \frac{\bu _{\alpha}}{T_{\alpha}}\right)  \cdot
%		\hat{\mathbf{M}}_{\alpha}+\left(  \frac{1}{T_{\alpha}}-\frac{1}{T_{\beta}}\right)  \hat
%		{e}_{\alpha}\geq0, \label{ResIneqNmix2}%
%		\end{equation}%
%\end{itemize}	
\noindent Next, we consider an {\it inert mixture} with zero chemical reactions:
\[  \tau_\alpha =0. \]
In this case, \eqref{New-1-1}, \eqref{ProdTransf}  and \eqref{New-3} become
\begin{align}
\begin{aligned} \label{New-5}
& {\bf M}_\alpha   = \hat{{\bf M}}_{\alpha}, \quad e_{\alpha}    =\hat{{\bf M}}_{\alpha} \cdot \bv +\hat{e}_{\alpha}, \\
&  \sum_{\alpha =1}^{n}\mathbf{\hat{M}}_{\alpha }=\mathbf{0},  \quad \sum_{\alpha =1}^{n}\hat{e}_{\alpha }=0,\\
&  \Sigma=\sum_{\beta=1}^{n}\left(
 - \frac{{\bu _{\beta}}}{T_\beta}\cdot\hat{\mathbf{M}}_{\beta}+\frac{1}{T_\beta}%
\hat{e}_{\beta}\right) \geq 0.
\end{aligned} 
\end{align} 
In particular, we can rewrite the last relation $\eqref{New-5}_5$  using \eqref{ResIneqNmix} as 
\begin{equation} \label{New-6}
\Sigma = \sum_{b=1}^{n-1}\left\{ \left(
\frac{\bu _{n}}{T_{n}}- \frac{\bu _{b}}{T_{b}}\right)  \cdot
\hat{\mathbf{M}}_{b}+\left(  \frac{1}{T_{b}}-\frac{1}{T_{n}}\right)  \hat
{e}_{b}\right\}\geq 0.
\end{equation}
Note that the crucial point for a possible difference between different models lies in the different choices of $\hat{\mathbf{M}}_\alpha$ and $\hat{e}_\alpha$ in such a way that the entropy inequality $\eqref{New-5}_5$ or equivalently \eqref{New-6} holds. 
 
\subsection{The PB-CS model} \label{sec:2.3} 
In this subsection, we discuss a formal derivation of a flocking model for thermo-mechanical  Cucker-Smale particles. Using a phenomenological theory discussed in Section \ref{sec:2.2}, we construct $(\hat{\mathbf{M}}_{b}, \hat{e}_b)$ such that the global entropy inequality \eqref{New-6} holds as a quadratic form which is typical of irreversible thermodynamics \cite{Simic1}:
	\begin{equation}\label{Mbeb}
	\hat{\mathbf{M}}_{b}=\frac{1}{n}\sum_{a=1}^{n-1}\psi _{a b}\left( \frac{\bu _{n}%
}{T_{n}}-\frac{\bu _{a}}{T_{a}}\right),~~
\hat{e}_{b}=\frac{1}{n}\sum_{a=1}^{n-1}\theta _{a b}\left(\frac{1}{T_{a}}-\frac{1}{%
	T_{n}}\right),~~b \in [n-1]. 
\end{equation}
If  $(\psi _{a b})$ and $(\theta _{a b})$ are symmetric and positive definite, the relations \eqref{New-6} and \eqref{Mbeb} yield that $\Sigma$ is positive in non equilibrium:
\begin{equation}
\Sigma = \sum_{a,b=1}^{n-1}  \left\{\psi _{a b}\left( \frac{\bu _{n}%
}{T_{n}}-\frac{\bu _{a}}{T_{a}}\right) \left(
\frac{\bu _{n}}{T_{n}}- \frac{\bu _{b}}{T_{b}}\right)  +\theta _{a b}\left(\frac{1}{T_{a}}-\frac{1}{%
	T_{n}}\right)\left(  \frac{1}{T_{b}}-\frac{1}{T_{n}}\right)  \right\}> 0.\label{entropiaxy}
	 \end{equation}
In \cite{Ha_Ruggeri}, the authors transform $(n-1) \times (n-1)$ matrices $(\psi_{ij})$ and $(\theta_{ij})$ into  $n \times n$ matrices $(\phi_{\alpha\beta})$ and $(\zeta_{\alpha \beta})$: 
\begin{align}
	\begin{aligned} \label{psi-phi}
		&\phi_{ij} :=-\psi_{ij}, \quad \forall \,  i\neq j \in [n-1], \\
		& \phi_{in} :=\phi_{ni}=\sum_{j=1}^{n-1}                    \psi_{ij}, \quad \forall \,  i \in [n-1], \\
		&\phi_{\alpha \alpha} := \text{arbitrary  value}~\forall~ \alpha \in [n].
	\end{aligned}
\end{align}
Similarly, we can define $(\zeta_{\alpha\beta})$ from $(\theta_{\alpha \beta})$. Moreover, if all components of $(\phi _{\alpha\beta})$ and $(\zeta _{\alpha\beta})$ are positive, then the matrices  $(\psi_{ij})$ and $(\theta_{ij})$  are  positive definite (see \cite{Ha_Ruggeri}). Then, the relations  \eqref{Mbeb} can be rewritten in terms of the new matrices $(\phi_{\alpha \beta})$ and $(\zeta_{\alpha \beta})$:
	\begin{equation}\label{prodmin}
\hat{\mathbf{M}}_{\alpha}=\frac{1}{n}\sum_{\beta=1}^{n}\phi _{\alpha\beta}\left( \frac{\bu _{\beta}%
	}{T_{\beta}}-\frac{\bu _{\alpha}}{T_{\alpha}}\right),  \quad
	 \hat{e}_{\alpha}=\frac{1}{n}\sum_{\beta=1}^{n}\zeta _{\alpha\beta}\left(\frac{1}{T_{\alpha}}-\frac{1}{%
		T_{\beta}}\right).
	\end{equation}
 Conversely, one may return from \eqref{prodmin} to  \eqref{Mbeb} via the inverse transformation \cite{Ha_Ruggeri} as well:
\begin{align} \label{phi-psi}
\begin{split}
    &\psi_{ij}=-\phi_{ij}, \quad \forall \,  i\neq j \in [n-1], \\
&\psi_{ii}=\sum_{\beta\neq i=1}^{n}\phi_{i\beta},\,\,\, \forall \,  i \in [n-1].
\end{split}
\end{align}
 We can also express $\Sigma$ using the new matrices $(\phi_{\alpha \beta})$ and $(\zeta_{\alpha \beta})$. \newline
 
 On the other hand, it follows from $\eqref{New-5}_5$ and \eqref{prodmin} that 
\[
	\Sigma=\frac{1}{n}\sum_{\alpha,\beta=1}^{n}\phi_{\alpha\beta} \left\{
	- \frac{{\bu _{\alpha}}}{T_\alpha}\cdot\left( \frac{\bu _{\beta}%
	}{T_{\beta}}-\frac{\bu _{\alpha}}{T_{\alpha}}\right)\right\}+\frac{1}{n}\sum_{\alpha,\beta=1}^{n} \frac{\zeta_{\alpha\beta}}{T_\alpha} \left( \frac{1%
}{T_{\alpha}}-\frac{1}{T_{\beta}}\right).
\]
Again, we use the index exchange transformation $\alpha~\leftrightarrow~\beta$ and  the symmetries of $\phi_{\alpha \beta}$ and $\zeta_{\alpha \beta}$ to find 
	\begin{align}\label{Sigma2}
		\Sigma=\frac{1}{2n}\sum_{\alpha,\beta=1}^{n}\phi_{\alpha\beta}   
		\left( \frac{\bu _{\beta}%
		}{T_{\beta}}-\frac{\bu _{\alpha}}{T_{\alpha}}\right)^2+\frac{1}{2n}\sum_{\alpha,\beta=1}^{n}\zeta_{\alpha\beta}   \left( \frac{1%
		}{T_{\alpha}}-\frac{1}{T_{\beta}}\right)^2 >0.
	\end{align}
 It is simple to verify that the expressions \eqref{entropiaxy} and \eqref{Sigma2} coincide by the relations \eqref{psi-phi} and \eqref{phi-psi}. In what follows, we assume that all components  of $(\phi_{\alpha\beta}) $ and $(\zeta_{\alpha\beta})$ are positive and therefore the matrices $(\psi_{ij})$ and $(\theta_{i,j})$ are symmetric and positive definite \cite{Ha_Ruggeri}. 

Note that the converse relation is not valid, in general, i.e. even if the matrices $(\psi_{ij})$ and $(\theta_{ij})$ are symmetric and positive definite, the symmetric matrices $(\phi_{\alpha\beta})$ and $(\zeta_{\alpha\beta})$ may not have all positive components. The entropy production \eqref{entropiaxy} or equivalently \eqref{Sigma2} is zero if and only if the mixture lies in an equilibrium state:
\[ \bu_\alpha={\bf 0} \quad \mbox{and} \quad  T_\alpha = T_0, \quad \forall~\alpha \in [n]. \]
System \eqref{Eulero} equipped with \eqref{constitutive} and \eqref{prodmin}  becomes a symmetric hyperbolic one by choosing field variables as the main field \eqref{MF_model1} \cite{Simic1}.  We set $S_\alpha$ and ${\bm \Phi}_\al$ to satisfy the Gibbs equation:
\begin{equation}\label{Gibbs}
    T_\alpha dS_\alpha = d\varepsilon_\al -\frac{p_\al}{\rho_\al^2} d\rho_\al, \quad   {\bm \Phi}_\al = \rho_\al S_\al \bv_\al.
\end{equation}
Entropy production $\Sigma$ is given by \eqref{entropiaxy}. This result follows from the symmetrization procedure by Ruggeri and Strumia \cite{RS}, when a hyperbolic system has an entropy law with a convex entropy. The consequence of having a symmetric system is to guarantee of a local well-posedness of the Cauchy problem. Moreover, so-called the K-condition \cite{Kawa} was satisfied, and therefore global smooth solutions exist for sufficiently small initial data \cite{BookNew}. 

Now, we assume the spatial homogeneity of observables and combine \eqref{Eulero} and \eqref{prodmin} to write down a phenomenological theory-based model for thermo-mechanical CS particles:
\begin{equation}
\begin{cases} \label{TCS-1}
\displaystyle \frac{d\mathbf{x}_\alpha}{dt} = \mathbf{v}_\alpha, \quad \alpha \in [n], \\
\displaystyle \frac{d\mathbf{v}_\alpha}{dt} = \frac{\kappa_1}{n} \sum_{\beta=1}^{n} \phi_{\alpha \beta} \Big( \frac{\mathbf{v}_\beta - \mathbf{v}}{T_\beta} -  \frac{\mathbf{v}_\alpha - \mathbf{v}}{T_\alpha} \Big), \\
\displaystyle \frac{d}{dt}\left(T_\alpha + {{\frac{1}{2}  {\mathbf{v}}_\alpha^2}}\right) = \frac{\kappa_2}{n}  \sum_{\beta=1}^{n} \zeta_{\alpha \beta} \Big( \frac{1}{T_\alpha} - \frac{1}{T_\beta}   \Big) \\
\displaystyle \hspace{3cm} +  \frac{\kappa_1}{n} \sum_{\beta=1}^{n} \phi_{\alpha \beta} \Big( \frac{\mathbf{v}_\beta - \mathbf{v}}{T_\beta} -  \frac{\mathbf{v}_\alpha - \mathbf{v}}{T_\alpha} \Big) \cdot \mathbf{v}.      
\end{cases}
\end{equation}
\vspace{0.2cm}

Before we close this section, we briefly comment on the advantages and disadvantages of the phenomenological theory of gas mixture  as follows: \newline
	\begin{itemize}
		\item \textbf{Advantage}: The validity is for any mixture of gases: rarefied polytropic in which internal energy $\varepsilon_\alpha$ is linear in temperature, rarefied non-polytropic gases in which  $\varepsilon_\alpha$ is non-linear in temperature, dense gases in which $\varepsilon_\alpha$ depending on not only temperature but also the density. 
		\vspace{0.2cm}
		
			\item \textbf{Disadvantage}: The matrices $(\phi _{\alpha\beta})$ and $(\zeta _{\alpha\beta})$  are not determined a priori, and in principle, the expression \eqref{prodmin} that render a quadratic form for \eqref{New-5} is not the only possible choice for inequality \eqref{New-5}.
	\end{itemize}

\section{A kinetic theory of Mixtures} \label{sec:3}
\setcounter{equation}{0}
In this section, we discuss a kinetic theory of mixtures which is parallel to the presentations in the previous section. \newline 

Let $f_\alpha\equiv f_\alpha({\bf x},{\bm \xi}, t)$ be the velocity distribution function of species $\alpha$ and ${\bm \xi}\equiv(\xi_i)$ is the molecular (or microscopic) velocity. Then, the starting point is the system of the Boltzmann equations for a gas mixture:
	\begin{equation}\label{falpfa}
	\frac{\partial f_\alpha}{\partial t} +\sum_{i=1}^n\xi_i \frac{\partial f_\alpha}{\partial x_i} = Q_\alpha :=\sum_{\beta =1}^n Q_{\alpha\beta}( f_\alpha, f_{\beta}), \quad \alpha \in [n],
	\end{equation}
or equivalently, 
\begin{align*}
& \partial_t f_1 + {\bm \xi} \cdot \nabla_{\bf x} f_1 = Q_{11}(f_1,f_1) + Q_{12}(f_1,f_2) +\cdots + Q_{1n}(f_1,f_n), \cr
& \vdots \hspace{3cm} \vdots \hspace{3cm} \vdots \cr
& \partial_t f_n + {\bm \xi} \cdot \nabla_{\bf x} f_n = Q_{n1}(f_n,f_1) + Q_{n2}(f_n,f_2) +\cdots + Q_{nn}(f_n,f_n),
\end{align*}
where  $Q_{\alpha\beta}$ is the Boltzmann collision operator satisfying the following relations:
\[
\sum_{\al=1}^{n} \int_{ \mathbb{R}^3} Q_{\al} \left(
\begin{array}{c}
1\\ {\bm \xi}\\ {\bm \xi^2}
\end{array}
\right)
d{\bm \xi}  = 0, \qquad \sum_{\al=1}^{n} \int_{\mathbb{R}^3} Q_{\al} \ln f_{\al} d{\bm \xi}  \leq 0.
\]
Different types of collisional operators were also discussed in the literature for monatomic and polyatomic gases. For the variants of BGK-type models, we refer to  Pirner review \cite{Pirner}. 
In what follows, we consider the simple BGK-type model introduced by Andries, Aoki, and Perthame \cite{AAP} with the collision operator:
\[
Q_\alpha = \nu_\alpha (f^M_\alpha - f_\alpha), \quad \alpha \in [n],
\]
where $\nu_\alpha$ is a collision frequency and $f^M_\alpha$ is a suitable Maxwellian whose explicit form is not needed at this point. \newline

\noindent Next, we set velocity moments up to second order:
\begin{align*}
& \rho_{\al}=m_{\al}\int_{\mathbb{R}^3}f_{\al}d{\bm \xi}, \quad \rho_{\al}\bv_{\al}=m_{\al}\int_{\mathbb{R}^3}{\bm \xi}  f_{\al}d{\bm \xi} , \cr
&E_{\al}=\frac{1}{2}\rho_{\al}|\bv_{\al}|^2+\rho_\alpha \varepsilon_{\al}=
\frac{m_{\al}}{2}\int_{\mathbb{R}^3}|{\bm \xi}|^2f_{\al}d{\bm \xi},
\end{align*}
where $\varepsilon_{\al}$ is the internal energy whose explicit form takes the following form: 
\begin{equation} \label{New-7}
\varepsilon_{\al}=\frac{m_{\al}}{2 \rho_\alpha }\int_{\mathbb{R}^3}|{\bm \xi-\bv_{\al}}|^2f_{\al}d{\bm \xi}.
\end{equation}
For a monatomic gas, internal energy $\varepsilon_\alpha$ is given as follows \cite{AAP}:
\begin{equation}\label{emono}
\varepsilon_{\alpha } = \frac{3}{2}\frac{k_B}{m_\alpha}T_\alpha,
\end{equation}
where $k_B$ is the Boltzmann constant.
\subsection{Production terms} \label{sec:3.1}
In this subsection, we study the production terms using the kinetic theory of mixtures. In \cite{AAP} (see also \cite{BMS}), the authors considering the first 
five moments for Eulerian gases obtained the same left-hand side as in the phenomenological system \eqref{RT_model}, whereas the production terms on the right-hand side are given explicitly in terms of macroscopic observables:

\begin{align}
\begin{aligned} \label{NNN-1}
\tau_\alpha & =   \sum_{\beta=1}^n\int_{\mathbb{R}^3}  m_\alpha   Q_{\alpha \beta}   d{\bm \xi}=0, \\
  {\bf M}_\alpha &=  \sum_{\beta=1}^n\int_{\mathbb{R}^3} m_\alpha  {\bm \xi}  Q_{\alpha \beta}   d{\bm \xi} = 
 \sum_{\beta=1}^n\int_{\mathbb{R}^3} m_\alpha  ({\bm \xi} -\bv_\alpha)  Q_{\alpha \beta}   d{\bm \xi} \\
 &= \sum_{\beta=1}^n\frac{2\rho_\alpha \rho_\beta }{m_\alpha +m_\beta } \chi_{\alpha \beta}     (\bv_\beta -\bv_\alpha ), \\
 {e}_\alpha &=   \sum_{\beta=1}^n\int_{\mathbb{R}^3} \frac{m_\alpha }{2}|{\bm \xi}|^2Q_{\alpha \beta}   d{\bm \xi} = 
\sum_{\beta=1}^n\int_{\mathbb{R}^3} \frac{m_\alpha }{2}|{\bm \xi} -\bv_\alpha |^2Q_{\alpha \beta}   d{\bm \xi}
+ \bv_\alpha \cdot  {\bf M}_\alpha \\
& = \sum_{\beta=1}^n \frac{2\rho_\alpha \rho_\beta }{(m_\alpha +m_\beta )^2} \chi_{\alpha \beta}   
\left\{3k_B(T_\beta - T_\alpha) +m_\beta |\bv_\alpha -\bv_\beta |^2 \right\}+ \bv_\alpha \cdot  {\bf M}_\alpha  \\
&= \sum_{\beta=1}^n \frac{2\rho_\alpha \rho_\beta }{(m_\alpha +m_\beta )^2} \chi_{\alpha \beta}    
\left\{ 3k_B(T_\beta -T_\alpha) +(m_\alpha \bv_\alpha +m_\beta \bv_\beta)(\bv_\beta -\bv_\alpha )\right\},
\end{aligned}
\end{align}
where $\chi_{\al\beta}$ is the positive and symmetric interaction coefficient whose expression can be found in \cite{AAP}.
The  previous expressions for the production terms satisfy the Galilean invariance \eqref{ProdTransf}:
\begin{align}
\begin{aligned}  \label{New-8}
	&{\bf M}_\alpha=  \hat{{\bf M}}_\alpha=
	\sum_{\beta=1}^n\frac{2\rho_\alpha \rho_\beta \chi_{\alpha \beta}}{m_\alpha +m_\beta }   (\bu_\beta -\bu_\alpha ), \quad  {e}_\alpha = \hat{e}_\alpha + \bv \cdot \hat{{\bf M}}_\alpha,\\
	& \hat{e}_\alpha =
	\sum_{\beta=1}^n \frac{2\rho_\alpha \rho_\beta  \chi_{\alpha \beta}}{(m_\alpha +m_\beta )^2}   
	 \Big \{ 3k_B(T_\beta -T_\alpha) +(m_\alpha \bu_\alpha +m_\beta \bu_\beta)(\bu_\beta -\bu_\alpha ) \Big \}. 
\end{aligned}
\end{align}
These production terms also satisfy the global requirement \eqref{New-5}.
\subsection{The KB-CS model} \label{sec:3.2}
In this subsection, we study the mixture of Eulerian monatomic gas \eqref{Eulero} with production terms \eqref{New-8} and caloric and thermal equation of state \eqref{constitutive} for rarefied monatomic gases:
\begin{equation}\label{mono}
    p_\alpha=\frac{k_B}{m_\al} \rho_\al T_\al, \quad \varepsilon_\al = \frac{3}{2}\frac{k_B}{m_\al} T_\alpha, \quad \alpha \in [n].
\end{equation} 
Even if there exists an $H$-theorem for the kinetic level, when we truncate the infinite moments associated with the Boltzmann equation, we cannot make sure whether the solutions to the truncated system \eqref{RT_model} satisfy an entropy principle or not. This is typical to all Rational Extended Thermodynamics in which the moments are truncated and closed using some universal principle like the Maximum Entropy Principle (MEP) (see \cite{BookNew} and references therein). \newline

Next, we verify that the production terms \eqref{New-8}  are compatible with an entropy principle.
\begin{theorem} \label{T3.1}
The production terms \eqref{New-8} satisfy the entropy inequality \eqref{New-5}, and any solution to system  \eqref{Eulero} with \eqref{New-8}, \eqref{mono} satisfies the entropy principle as well.
\end{theorem}
\begin{proof}  We set 
\begin{equation} \label{New-9}
b_{\alpha\beta} =\frac{2\rho_\alpha \rho_\beta }{(m_\alpha +m_\beta )^2 } \chi_{\alpha \beta}, \quad \alpha, \beta \in [n].
\end{equation}
Then, the matrix $(b_{\alpha \beta})$ is symmetric with respect to the index exchange transformation $\alpha~\leftrightarrow~\beta$. Now, we substitute \eqref{New-8} into $\eqref{New-5}_5$ to obtain
\begin{align}
\begin{aligned} \label{New-10}
 \Sigma &= \sum_{\beta=1}^{n}\left(
 - \frac{{\bu _{\beta}}}{T_\beta}\cdot\hat{\mathbf{M}}_{\beta}+\frac{1}{T_\beta} \hat{e}_{\beta}\right) \\
 &=  \sum_{\alpha,\beta=1} ^{n}   b_{\alpha\beta} \Big ( -\frac{\bu _\alpha}{T_\alpha} \cdot (\bu _\beta-\bu _\alpha)(m_\alpha +m_\beta) \\
 & \hspace{0.5cm} +\frac{1}{T_\alpha}\left\{ 3k_B(T_\beta -T_\alpha) +(m_\alpha \bu_\alpha +m_\beta \bu_\beta) \cdot (\bu_\beta -\bu_\alpha )\right\} \Big ) \\
 &= \sum_{\alpha,\beta=1} ^{n}   b_{\alpha\beta}\left\{\frac{m_\beta}{T_\alpha}|\bu_\beta-\bu_\alpha|^2+3 k_B\left(\frac{T_\beta}{T_\alpha}-1\right)
 \right\} \\
 &= {\mathcal I}_{11} + {\mathcal I}_{12}.
 \end{aligned}
 \end{align}
 Below, we estimate the term ${\mathcal I}_{1i}$ one by one. \newline
 
 \noindent $\bullet$~(Estimate of ${\mathcal I}_{11}$): We use the index exchange transformation $\alpha~\leftrightarrow~\beta$ to find 
 \begin{align}
 \begin{aligned} \label{New-10-1}
 {\mathcal I}_{11} &= \sum_{\alpha,\beta=1} ^{n}   b_{\alpha\beta} \frac{m_\beta}{T_\alpha}|\bu_\beta-\bu_\alpha|^2 = \sum_{\alpha,\beta=1} ^{n}   b_{\alpha\beta} \frac{m_\alpha}{T_\beta}|\bu_\beta-\bu_\alpha|^2 \\
 &= \frac{1}{2}  \sum_{\alpha,\beta=1} ^{n}   b_{\alpha\beta}  \Big(  \frac{m_\beta}{T_\alpha} +   \frac{m_\alpha}{T_\beta}  \Big)  |\bu_\beta-\bu_\alpha|^2 \\
 &=  \sum_{\alpha,\beta=1} ^{n}  \frac{b_{\alpha\beta}}{2 T_\alpha T_\beta} ( m_\alpha T_\alpha + m_\beta T_\beta)   |\bu_\beta-\bu_\alpha|^2.
 \end{aligned}
 \end{align}
 
 \vspace{0.2cm}
 
 \noindent $\bullet$~(Estimate of ${\mathcal I}_{12}$): Similarly, we have 
 \begin{align}
 \begin{aligned} \label{New-10-2}
 {\mathcal I}_{11} &= 3 k_B  \sum_{\alpha,\beta=1} ^{n}   b_{\alpha\beta}  \left(\frac{T_\beta}{T_\alpha}-1\right) =  3 k_B  \sum_{\alpha,\beta=1} ^{n}   b_{\alpha\beta}  \left(\frac{T_\alpha}{T_\beta}-1\right) \\
 &= \frac{ 3 k_B}{2}  \sum_{\alpha,\beta=1} ^{n}   b_{\alpha\beta} \Big( \frac{T_\beta}{T_\alpha} + \frac{T_\alpha}{T_\beta} -2  \Big) = \sum_{\alpha,\beta=1} ^{n}  \frac{3 k_B b_{\alpha\beta}}{2T_\alpha T_\beta} |T_\beta-T_\alpha|^2. 
 \end{aligned}
 \end{align}
 Finally, we combine \eqref{New-10}, \eqref{New-10-1} and \eqref{New-10-2} to find in non equilibrium
\begin{equation}\label{SigmaK}
\Sigma =  \sum_{\alpha,\beta=1}^{n}   \frac{b_{\alpha\beta}}{2T_\alpha T_\beta}  \Big ( (m_\alpha T_\alpha + m_\beta T_\beta) |\bu_\beta-\bu_\alpha|^2  + 3 k_B  |T_\beta-  T_\alpha|^2 \Big) > 0.
\end{equation}
Note that in equilibrium
\[ \Sigma = 0 \quad \iff \quad \bu_\beta = \bu_\alpha={\bf 0} \quad \mbox{and} \quad T_\alpha = T_\beta = T_0. \]
From  Gibbs' equation \eqref{Gibbs}, we have the supplementary entropy law \eqref{EP} with
\[
\rho S = \sum_{\al =1}^n \rho_\al S_\al, \quad {\bm \Phi} = \sum_{\al =1}^n \rho_\al S_\al \bv_\al, \quad 
S_\al = \log \left(\frac{T_\al^\frac{3}{2}}{\rho_\al}\right)^\frac{k_B}{m_\al},
\]
and entropy production given by \eqref{SigmaK}.  
\end{proof} 

\vspace{0.5cm}

Now, we combine \eqref{Eulero} and \eqref{New-8} to write down the kinetic theory-based model for thermo-mechanical CS particles:
\begin{equation}
\begin{cases} \label{New-11}
\displaystyle \frac{d\mathbf{x}_\alpha}{dt} = \mathbf{v}_\alpha, \quad \alpha \in [n], \\
\displaystyle \frac{d\mathbf{v}_\alpha}{dt} = \sum_{\beta=1}^n b_{\alpha \beta} (m_\alpha + m_\beta) (\bu_\beta -\bu_\alpha ), \\
\displaystyle \frac{d}{dt}\left(T_\alpha + {\frac{1}{2}  {|\bv_\alpha|^2}}\right) =  \sum_{\beta=1}^n b_{\alpha \beta} 
\left\{ 3k_B(T_\beta -T_\alpha) +(m_\alpha \bu_\alpha +m_\beta \bu_\beta)(\bu_\beta -\bu_\alpha )\right\} \\
\displaystyle \hspace{2cm} + \sum_{\beta=1}^n b_{\alpha \beta} (m_\alpha + m_\beta) (\bu_\beta -\bu_\alpha ) \cdot \bv.
\end{cases}
\end{equation}

Before we close this section, we also consider the advantages and disadvantages of the modeling based on kinetic theory. In some sense, the advantages and disadvantages of the KB-CS model are orthogonal to those of the PB-CS model:
\begin{itemize}
	\item \textbf{Advantage}: The matrices in the production terms \eqref{NNN-1}  are explicit, once we know the matrix $\chi_{\alpha \beta}$.
	\vspace{0.1cm}
	\item \textbf{Disadvantage}: The model is valid only for rarefied gases. In particular in this presentation for monatomic rarefied gases.
\end{itemize}

\section{Two normalized particle models for flocking} \label{sec:4}
\setcounter{equation}{0}
In this section, we study the comparison of two normalized particle models \eqref{TCS-1} and \eqref{New-11} by comparing the production terms resulting from phenomenological theory-based approach and kinetic theory-based approach. Since the coefficients in the right-hand side of  \eqref{TCS-1} and \eqref{New-11} depend on the matrices $(\phi_{\alpha \beta}), (\zeta_{\alpha \beta})$ and $(b_{\alpha \beta})$, we need suitable normalization for the comparison of two-particle models. These will be the content of discussions in the sequel. 

\subsection{Normalized productions terms} \label{sec:4.1}
First, we recall that the production terms \eqref{prodmin} and \eqref{New-8}:
\begin{equation}
\begin{cases}  \label{New-12}
\displaystyle \hat{\mathbf{M}}_{\alpha}=\frac{1}{n}\sum_{\beta=1}^{n}\phi _{\alpha\beta}\left( \frac{\bu _{\beta}%
	}{T_{\beta}}-\frac{\bu _{\alpha}}{T_{\alpha}}\right), \\
\displaystyle  \hat{e}_{\alpha}=\frac{1}{n}\sum_{\beta=1}^{n}\zeta _{\alpha\beta}\left(\frac{1}{T_{\alpha}}-\frac{1}{T_{\beta}}\right),
\end{cases}
\end{equation}
and 
\begin{equation}
\begin{cases} \label{New-13}
\displaystyle  \hat{{\bf M}}_\alpha= \frac{1}{n}
	\sum_{\beta=1}^n\frac{2n \rho_\alpha \rho_\beta }{m_\alpha +m_\beta } \chi_{\alpha \beta}  (\bu_\beta -\bu_\alpha), \\
\displaystyle  \hat{e}_\alpha = \frac{1}{n}
	\sum_{\beta=1}^n \frac{2n \rho_\alpha \rho_\beta  \chi_{\alpha \beta}}{(m_\alpha +m_\beta )^2}   
	\Big(  3k_B(T_\beta -T_\alpha) +(m_\alpha \bu_\alpha +m_\beta \bu_\beta)(\bu_\beta -\bu_\alpha ) \Big). 
\end{cases}
\end{equation}
At first glance,  production terms \eqref{New-12} and \eqref{New-13}  look completely different.  However, if we choose a well-prepared ansatz for $\phi_{\alpha \beta}$ and $\zeta_{\alpha \beta}$, then production terms can coincide at the first order. To see this, we consider the situation in which temperatures are close to the common temperature $T_0$ and diffusion velocities are close to zero: 
\[ |T_\alpha - T_0|  \ll 1 \quad \mbox{and} \quad  |\bu_\alpha |  \ll 1, \quad \alpha \in [n].  \]
In this case, one has 
\[ \frac{\bu _{\beta}}{T_{\beta}}-\frac{\bu _{\alpha}}{T_{\alpha}} \approx  \frac{1}{T_0} (\bu _{\beta} - \bu_\alpha), \quad \frac{1}{T_{\alpha}}-\frac{1}{T_{\beta}} \approx \frac{T_\beta - T_\alpha}{T_0^2}.   \]
If we choose 
\begin{equation}\label{lineare}
\phi_{\alpha\beta}= \frac{2n T_0\rho_\alpha \rho_\beta }{m_\alpha +m_\beta } \chi_{\alpha \beta}, \quad \zeta_{\alpha\beta}= \frac{6n k_B T_0^2\rho_\alpha \rho_\beta }{(m_\alpha +m_\beta )^2} \chi_{\alpha \beta}, 
\end{equation} 
then the production terms \eqref{New-12} and \eqref{New-13} coincide at the first-order. \newline

\noindent For the non-linear case, we express $\chi_{\alpha \beta}$ using $\phi_{\alpha \beta}$ from the relation $\eqref{lineare}_1$:  
\begin{equation} \label{New-14}
\chi_{\alpha \beta} = \frac{(m_\alpha + m_\beta)}{2n T_0 \rho_\alpha \rho_\beta} \phi_{\alpha \beta},
\end{equation}
and we substitute this relation into $\eqref{lineare}_2$ to see
\begin{equation} \label{New-15}
\zeta_{\alpha \beta} = \frac{3 k_B T_0}{m_\alpha + m_\beta} \phi_{\alpha \beta}. 
\end{equation}
Now, we substitute \eqref{New-14} into \eqref{New-13} and \eqref{lineare} to find production terms and entropy production term in terms of $(\phi_{\alpha \beta})$. \newline

\noindent $\bullet$~Case A (Production terms based on phenomenological theory):~we combine \eqref{New-12} and \eqref{New-15} to get 
\begin{align} 
\begin{aligned} \label{Macros}
&\hat{\mathbf{M}}_{\alpha}=\frac{1}{n}\sum_{\beta=1}^{n}\phi _{\alpha\beta}\left( \frac{\bu_{\beta}}{T_{\beta}}-\frac{\bu _{\alpha}}{T_{\alpha}}\right), \quad \hat{e}_{\alpha}=\frac{3 k_B T_0}{n}\sum_{\beta=1}^{n}\frac{\phi _{\alpha\beta}}{m_\alpha + m_\beta}\left(\frac{1}{T_{\alpha}}-\frac{1}{T_{\beta}}\right), \\
& \Sigma= \frac{1}{2n}\sum_{\alpha, \beta=1}^{n}\phi _{\alpha\beta}\left\{ \left( \frac{\bu _{\beta}}{T_{\beta}}-\frac{\bu _{\alpha}}{T_{\alpha}}\right)^2 + \frac{3k_B T_0}{m_\alpha + m_\beta}\left(\frac{1}{T_{\alpha}}-\frac{1}{T_{\beta}}\right)^2 \right \} \geq 0.
\end{aligned}
\end{align}
		
\vspace{0.2cm}		
    
 \noindent $\bullet$~Case B (Production terms based on kinetic theory):~we combine \eqref{New-13} and \eqref{New-14} to see
\begin{align}
\begin{aligned}\label{Micros}
&\hat{{\bf M}}_\alpha=\frac{1}{n T_0} \sum_{\beta=1}^n  \phi_{\alpha \beta}  (\bu_\beta -\bu_\alpha ),  \\
& \hat{e}_\alpha =\frac{1}{n T_0}\sum_{\beta=1}^{n}\frac{\phi _{\alpha\beta}}{\left(m_\alpha + m_\beta\right)}
\left\{ 3k_B(T_\beta -T_\alpha) +(m_\alpha \bu_\alpha +m_\beta \bu_\beta)(\bu_\beta -\bu_\alpha )\right\},  \\
 & \Sigma= \frac{1}{2n T_0}\sum_{\alpha, \beta=1}^{n}\frac{\phi _{\alpha\beta}}{T_\alpha T_\beta(m_\alpha + m_\beta)}\left\{ 
 \left(m_\alpha T_\alpha + m_\beta T_\beta\right)
|\bu _{\beta} - \bu _{\alpha} |^2 +  3k_B |T_\beta-T_\alpha|^2\right\}. 
\end{aligned}
\end{align}
 \subsection{Spatially homogeneous processes} \label{sec:4.2} 
 Consider a spatially homogeneous flow in which observables $(\rho_\alpha, \rho_\alpha \bv_\alpha, e_\alpha)$ depend only on time so that flux terms are all zero. In this case,  system \eqref{Eulero} becomes 
\begin{equation} \label{flock1}
\frac{d \rho_\alpha}{dt} = 0, \quad   \frac{d (\rho_\alpha \bv _\alpha)}{dt} =\hat{{\bf M}}_\alpha, \quad  \frac{d}{dt} \left(\rho_\alpha\varepsilon_\alpha +\frac{1}{2}\rho_\alpha \bv_\alpha^2 \right) =\hat{{e}}_\alpha + \hat{{\bf M}}_\alpha \cdot \bv.
\end{equation}
For the whole mixture, we have
\[
\frac{d \rho }{dt}=0, \quad \frac{d (\rho \bv) }{dt}=0, \quad \frac{d}{dt} \sum_{\al=1}^n  \left(\rho_\alpha\varepsilon_\alpha +\frac{1}{2}\rho_\alpha \bv_\alpha^2 \right)=0.
\]
Now, without loss of generality, we take 
\begin{equation} \label{New-15-1}
\bv = 0.
\end{equation}
Then, one has 
\[ \bu_\alpha = \bv_\alpha - \bv = \bv_\alpha, \]
and the relations  \eqref{flock1} imply 
\begin{equation}\label{flock}
\frac{d \bx_\al}{dt} = \bu_\al, \quad  \rho_\alpha  \frac{d  \bu _\alpha}{dt} =\hat{{\bf M}}_\alpha, \quad \rho_\alpha \frac{d}{dt} \left(\frac{3}{2}k_B T_\alpha +\frac{1}{2} \bu_\alpha^2 \right) =\hat{{e}}_\alpha. 
\end{equation}
Any solution of \eqref{flock} satisfies the entropy law:
\begin{equation}\label{entplaw}
\rho \frac{d S}{dt}= \Sigma >0, \quad \text{with} \quad \rho S = \sum_{\al=1}^n \rho_\al \log \left(\frac{T_\al^\frac{3}{2}}{\rho_\al}\right)^\frac{k_B}{m_\al}.
\end{equation}
The quantities $\rho_\alpha$ and $\rho$ are constants, and production terms  $\hat{{\bf M}}_\alpha$, $\hat{e}_\alpha$ and $\Sigma$ are given by the phenomenological theory \eqref{Macros} and  kinetic theory from \eqref{Micros}.
Then, the diffusion velocities $\bu_\al$ satisfy
\[ \sum_{\al =1}^n \rho_\al \bu_\al(t) =0, \quad t \geq 0, \]
and we chose reference as the rest frame. Then the initial data must satisfy the conditions:
\begin{equation}\label{inu}
   \sum_{\alpha =1}^n \rho_\alpha \bx_\alpha(0) =0, \qquad   \sum_{\alpha =1}^n \rho_\alpha \bu_\alpha(0) =0.
\end{equation}
Recall that global energy is conserved. Thus, it is equal to the equilibrium characterized by 
\[ \bu_\alpha =0, \quad  T_\al = T_\infty. \]
Therefore, we can evaluate the equilibrium temperature
    $T_\infty$ by the initial data:
\begin{equation}\label{T0}
    \sum_{\alpha =1}^n \left( \frac{1}{2}\rho_\alpha u^2_\alpha(0) + \frac{3}{2}\frac{k_B}{m_\al}\rho_\al T_\alpha(0) \right)= T_\infty \sum_{\alpha =1}^n \frac{3}{2}\frac{k_B}{m_\alpha} \rho_\alpha .
\end{equation}
\subsection{Particle models for flocking} \label{sec:4.3}
In this subsection, we compare the qualitative analysis for the PB-CS  and the KB-CS models. For simplicity, we chose all the constants as
\[ \rho_\alpha =1, \quad  m_\alpha =1, \quad  k_B= \frac{2}{3}, \]
and we set 
\begin{equation} \label{New-16}
a_{\alpha \beta} = \frac{\phi_{\alpha \beta}}{T_0} = n\chi_{\alpha \beta}.
\end{equation}
With these choices \eqref{New-16}, 
we can write the entropy law \eqref{entplaw} as 
\[
n S =\sum_{\al=1}^n \ln T_\al.
\]
Moreover the condition of initial data \eqref{inu} and \eqref{T0} become
\begin{equation}\label{initialinf}
\sum_{\al=1}^n \bx_\al(0) =0, \quad \sum_{\al=1}^n \bu_\al(0) =0, \quad  \frac{1}{n} \sum_{\al=1}^n \left(T_\al(0)+\frac{1}{2} u_\al^2(0)\right) =T_\infty.
\end{equation}
Then the following quantities are conserved for both models.
\begin{equation}\label{conserved}
\sum_{\al=1}^n \bx_\al(t) =0, \quad 
\sum_{\al=1}^n \bu_\al(t) =0, \quad
\frac{1}{n} \sum_{\al=1}^n \left(T_\al(t)+\frac{1}{2} u_\al^2(t)\right) = T_\infty.
\end{equation}
The flocking models \eqref{TCS-1} and \eqref{New-11} can be rewritten as follows. \newline

\noindent $\bullet$~Case A: The PB-CS model reads as follows.
\begin{equation} \label{TCS}
\begin{cases}
\displaystyle \frac{d\bx_{\al}}{dt} = \bu_{\al}, \quad \alpha \in [n], \\
\displaystyle \frac{d\bu_{\al}}{dt} = \frac{T_0}{n}\sum_{\beta=1}^n a_{\alpha \beta}\left(\frac{\bu_{\beta}}{T_\beta}-\frac{\bu_{\al}}{T_\al}\right), \\
\displaystyle \frac{d}{dt} \left(T_{\al}+ \frac{1}{2}u_{\al}^2\right)  = \frac{T_0^2}{n}\sum_{\beta=1}^n a_{\alpha \beta}\left(\frac{1}{T_\al} -\frac{1}{T_\beta} \right).
\end{cases}
\end{equation}
Note that this model \eqref{TCS} satisfies the entropy law: 
\begin{equation}\label{S-TCS}
n \frac{dS}{dt} = \frac{T_0}{2n} \sum_{\alpha,\beta=1}^n a_{\al\beta}\left\{\left(\frac{\bu_\beta}{T_\beta} -\frac{\bu_\al}{T_\al}\right)^2 + T_0 \left(\frac{1}{T_\alpha}-\frac{1}{T_\beta}\right)^2 \right\} \geq0.
\end{equation}
\noindent $\bullet$~Case B: The KB-CS model reads as follows.
\begin{equation}  \label{KCS}
\begin{cases}
\displaystyle \frac{d\bx_{\al}}{dt} = \bu_{\al}, \quad \alpha \in [n], \\
\displaystyle \frac{d\bu_{\al}}{dt} = \frac{1}{n}\sum_{\beta=1}^n a_{\alpha \beta}\left(\bu_{\beta}-\bu_{\al}\right),  \\
\displaystyle \frac{d}{dt} \left(T_{\al}+ \frac{1}{2}u_{\al}^2\right)  = \frac{1}{n}\sum_{\beta=1}^n a_{\alpha \beta}\left\{\left(T_{\beta}+ \frac{1}{2}u_{\beta}^2\right)-\left(T_{\al}+ \frac{1}{2}u_{\al}^2\right)\right\}.
\end{cases}
\end{equation}
This also satisfies the entropy inequality: 
\begin{equation}\label{S-KCS}
n \frac{dS}{dt} = \frac{1}{2n} \sum_{\alpha,\beta=1}^n \frac{a_{\al\beta}}{T_\al T_\beta} \left\{(T_\al +T_\beta)\frac{|\bu_\beta - \bu_\al|^2}{2} +|T_\beta - T_\alpha|^2\right\} \geq 0.
\end{equation}

\subsection{Remarks on $T_0$ and $T_\infty$} \label{sec:4.4}
A delicate point is $T_0$  appearing only in the PB-CS model \eqref{TCS} and not in the KB-CS \eqref{KCS}. In principle, $T_0$ is an arbitrary constant equilibrium value of the temperature for which the linearized system \eqref{TCS} and \eqref{KCS} coincides (see Section \eqref{sec:4.1}). However, for homogeneous solutions, the only equilibrium temperature is $T_\infty$ due to the constancy of global energy equation \eqref{conserved}$_3$.  Now $T_0= T_\infty$ (that has the meaning of average and flocking temperature)  depends on the initial data \eqref{initial}$_3$. If we take arbitrary initial data, $T_0$ changes and system \eqref{TCS} changes every time depending on the initial data. This is not reasonable and therefore if we want to compare two models, the correct procedure seems that we first fix a priori $T_0$ as that of flocking temperature: 
\[ T_0=T_\infty. \]
This provides constraints for initial data. In summary, to compare two particle models, we need to choose well-prepared initial data satisfying the following constraints:
\begin{equation}\label{initial}
\sum_{\al=1}^n \bx_\al(0) =0, \quad \sum_{\al=1}^n \bu_\al(0) =0, \quad  \frac{1}{n} \sum_{\al=1}^n \left(T_\al(0)+\frac{1}{2} u_\al^2(0)\right) =T_0.
\end{equation}
and the flocking state will be also characterized by the prescribed pair:
\[ \bu_\alpha =0,  \quad T_\alpha=T_0, \quad \forall~\alpha \in [n]. \]
The relation $\eqref{initial}_3$  follows from the fact the (absolute) temperatures must be positive and the assertion that $u_\al^2(0)$ is also positive gives a strong limitations between initial temperatures and initial diffusion velocities. In particular, we have
\[ T_0 > \overline{T}(0), \]
where $\overline{T}(0)$ is the average initial temperature:
\begin{equation*}
	\overline{T}(0)= \frac{1}{n} \sum_{\al=1}^n   T_\al(0).
\end{equation*}
This condition implies that we cannot have small thermal diffusion and large mechanical diffusion. In fact if  $|T_\al(0)-T_0| = {\mathcal O}(\varepsilon),$ then one has 
\[  |\overline{T}(0) - T_0| = {\mathcal O}(\varepsilon). \]
 As a consequence of \eqref{initial}$_3$, one  has 
 \[ u_\al^2(0) = {\mathcal O}(\varepsilon). \]
\section{Quantitative estimates for particle CS models} \label{sec:5}
\setcounter{equation}{0}
In this section, we study three issues for the particle models introduced in previous section. More precisely, we deal with the following issues: 
\begin{itemize}
\item
Flocking estimate for the KB-CS model \eqref{KCS}.
\vspace{0.2cm}
\item
Convergence between two models \eqref{TCS} and \eqref{KCS}, when the initial velocity and temperature are small perturbation around flocking state.
\item
Comparison of the difference of dynamics between two models.
\end{itemize}
In the following two subsections, we consider the above items one by one. 
\subsection{Flocking dynamics for the KB-CS model} \label{sec:5.1}
In this subsection, we consider the flocking estimate for the KB-CS model \eqref{KCS}. For this, we consider the following two types of interaction matrix for $a_{\alpha \beta}$:
\begin{enumerate}
\item
(Type A):~Constant symmetric positive network topology:
\begin{equation} \label{NNN-2}
a_{\alpha \beta} = a_{\beta \alpha} \in \mathbb{R}, \quad \forall~\alpha, \beta \in [n], \quad  \min_{\alpha,\beta} a_{\al\beta} = \underline{a} > 0.
\end{equation}
\item
(Type B):~Metric dependent non-negative network topology:
\begin{equation} \label{NNN-3}
a_{\alpha \beta}  :=  \frac{1}{(1+|\bx_{\beta}-\bx_{\al}|^2)^{\lambda}}, \quad \mbox{for}~~0<\lambda\leq \frac{1}{2}. 
\end{equation}
\end{enumerate}
We set the minimum value of the function $a_{\alpha \beta}$ as $\phi$:
\begin{equation} \label{NNN-4}
\phi(t) := \min_{1\leq \al,\beta \leq n} \frac{1}{(1+|\bx_{\beta}-\bx_{\al}|^2)^{\lambda}}.
\end{equation}
For observables $\{ (\bx_\alpha, \bu_\alpha, T_\alpha) \}$, we define $\ell^2$-norms of position, velocity and energy fluctuation as follows:
\begin{align} 
\begin{aligned} \label{XVdef}
\mathcal{X} &:= \bigg(\sum_{\al=1}^n |\bx_{\al}|^2 \bigg)^{\frac{1}{2}}, \quad \mathcal{V} := \bigg(\sum_{\al=1}^n |\bu_{\al}|^2 \bigg)^{\frac{1}{2}}, \\
E_{\al} &:= T_{\al}+ \frac{1}{2}u_{\al}^2 - T_0, \quad  \mathcal{E} :=\bigg(\sum_{\al=1}^n |E_\alpha|^2 \bigg)^{\frac{1}{2}},
\end{aligned}
\end{align}
where $T_0$ is defined in $\eqref{initial}_3$.
For notational simplicity, we set 
\[ {\mathcal X}_0 := {\mathcal X}(0), \quad  {\mathcal V}_0 := {\mathcal V}(0), \quad {\mathcal E}_0 := {\mathcal E}(0).  \]
Note that the energy conservation law $\eqref{conserved}_3$ implies 
\begin{equation} \label{New-16-1}
 \sum_{\al=1}^nE_{\al}(t) =0, \quad t \geq 0.  
 \end{equation}
For later use, we recall the flocking estimate of the PB-CS model \eqref{TCS} in the following theorem. 
\begin{theorem}\label{T5.1}\cite{Ha_Ruggeri}
Suppose that network topology satisfies Type A conditions, and initial data satisfy \eqref{initial} and 
\begin{align}
\begin{aligned} \label{New-16-2}
&  |\bu_{\al}(0)| \leq \frac{\varepsilon}{2}, \quad  |T_{\al}(0)-T_0| \leq \frac{\varepsilon T_0}{2}, \\
& \sum_{\al=1}^n \Big ( \frac{|\bu_{\al}(0)|^2}{2} + |T_{\al}(0)-T_0|^2 \Big) \leq \frac{\varepsilon^2}{8},
\end{aligned}
\end{align}
for some positive constant $\varepsilon$, and let $(\bx_{\al},\bu_{\al},T_{\al})$ be the solution of system \eqref{TCS}. Then, we have
\begin{align}
\begin{aligned} \label{New-16-2.5}
|\bu_{\al}(t)|< \varepsilon,~~|T_{\al}(t)-T_0|< \varepsilon T_0,~~\mathcal{V}(t)  \leq  C \mathcal{V}_0 e^{-\underline{a} t},~~\mathcal{E}(t) \leq  C \mathcal{E}_0 e^{-\underline{a} t}, \quad t \geq 0.
\end{aligned}
\end{align}
\end{theorem}
\begin{remark}\label{R5.1}
If \eqref{New-16-2.5} holds, then we have 
\[  |\bx_\alpha(t) - \bx_\alpha(0) | \leq \frac{C {\mathcal V}_0}{\underline{a}}, \quad  |\bu_\alpha(t)  | \leq C \mathcal{V}_0 e^{-\underline{a} t}, \quad  |E_\alpha(t)| \leq C \mathcal{E}_0 e^{-\underline{a} t}, \quad t \geq 0. \]
\end{remark}

\vspace{0.5cm}

Next, we return to the flocking dynamics of \eqref{KCS}.
\noindent For later use, we set 
\begin{align}
\begin{aligned} \label{NNN-5}
C_{\lambda}(t) &:= \begin{cases} 
\displaystyle \int_0^t e^{-\frac{k}{1-2\lambda}((1+s)^{1-2\lambda}-1)}ds, \quad & \lambda \in (0, \frac{1}{2}), \\  
\displaystyle \int_0^t (1+s)^{-\frac{k}{2}} ds, \quad &\lambda=\frac{1}{2}, \end{cases} \\
\Lambda_0 &:= (\max\{1+2\mathcal{X}_0^2,2\mathcal{V}_0^2\})^{-\lambda} \leq 1.
\end{aligned}
\end{align}

\noindent Now, we are ready to state our first main result on the flocking behaviors of the model \eqref{KCS}. 

\begin{theorem} \label{T5.2}
%\begin{itemize}
%\item 
%\end{itemize}
Let $\{ (\bx_{\al},\bu_{\al},T_{\al}) \}$ be a global solution of system \eqref{KCS} with the initial data $\{ (\bx_{\al}(0),\bu_{\al}(0),T_{\al}(0)) \}$ satisfying \eqref{initial}. Then, one has the following assertions:
\begin{enumerate}
\item If the network topology $(a_{\alpha \beta})$ satisfies Type A condition \eqref{NNN-2},  then we have 
\[ \mathcal{V}(t) \leq \mathcal{V}_0 e^{-\underline{a} t}, \quad {\mathcal X}(t) \leq \mathcal{X}_0 + \frac{\mathcal{V}_0}{\underline{a}}, \quad  \mathcal{E}(t) \leq \mathcal{E}_0 e^{-\underline{a} t}, \quad t \geq 0. \]
\item
If the network topology $(a_{\alpha \beta})$ satisfies Type B condition \eqref{NNN-3}, then we have 
\begin{align*}
\begin{aligned}
& \mathcal{V}(t) \leq \mathcal{V}_0 \begin{cases} e^{-\frac{k}{1-2\lambda}((1+t)^{1-2\lambda}-1)}, \quad & \lambda \in (0, \frac{1}{2}), \\ 
(1+t)^{-k}, \quad &\lambda=\frac{1}{2},
\end{cases} \quad  \mathcal{X}(t) \leq \mathcal{X}_0 + C_{\lambda}(t)\mathcal{V}_0, \\
& \mathcal{E}(t) \leq \mathcal{E}_0 \begin{cases} e^{-\frac{k}{1-2\lambda}((1+t)^{1-2\lambda}-1)}, \quad & \lambda \in (0, \frac{1}{2}), \\ 
(1+t)^{-k}, \quad &\lambda=\frac{1}{2},
\end{cases} 
\end{aligned}
\end{align*}
where $C_{\lambda}(t)$ is defined in \eqref{NNN-5}.
\end{enumerate}
\end{theorem}
\begin{proof} 
\noindent (i)~Suppose that the network topology $(a_{\alpha \beta})$ satisfies Type A conditions \eqref{NNN-2}.  \\
\noindent $\bullet$~Case A (Estimates for ${\mathcal X}$ and ${\mathcal V}$):  By the Cauchy-Schwartz inequality and \eqref{XVdef}, we have
\begin{align*}
\frac{d\mathcal{X}^2}{dt} 
=2\sum_{\al=1}^n \bx_{\al} \cdot \bu_{\al}
\leq 2 \sqrt{\sum_{\al=1}^n|\bx_{\al}|^2}\sqrt{\sum_{\al=1}^n|\bu_{\al}|^2}   = 2\mathcal{X} \cdot \mathcal{V}.
\end{align*}
This yields
\begin{align}\label{XV1} 
\bigg|\frac{d\mathcal{X}}{dt} \bigg| \leq \mathcal{V}.
\end{align}
On the other hand, one has 
\begin{align}\label{VKCS}
\begin{aligned}
\frac{d\mathcal{V}^2}{dt} 
&= \frac{2}{n} \sum_{\al=1}^n\sum_{\beta=1}^n a_{\al\beta}\left(\bu_{\beta}-\bu_{\al}\right) \cdot  \bu_{\al}
= -\frac{1}{n} \sum_{\al=1}^n\sum_{\beta=1}^n a_{\al\beta} |\bu_{\beta}-\bu_{\al}|^2  \\
&\leq -\frac{\underline{a}}{n} \sum_{\al=1}^n\sum_{\beta=1}^n  (|\bu_{\beta}|^2 -2 \bu_\beta \cdot \bu_\al +|\bu_{\al}|^2) \\
&=  -\frac{\underline{a}}{n} \sum_{\al=1}^n\sum_{\beta=1}^n  (|\bu_{\beta}|^2 +|\bu_{\al}|^2) \leq - 2 \underline{a} \mathcal{V}^2 ,
\end{aligned}
\end{align}
where we used the symmetric relation $a_{\alpha \beta} = a_{\beta \alpha}$ and $\eqref{conserved}_2$.  This yields 
\begin{align}\label{VexpKCS}
\mathcal{V}(t) \leq \mathcal{V}_0 e^{-\underline{a} t}.
\end{align}
Now, we use \eqref{XV1} and  \eqref{VexpKCS} to see
\begin{equation} \label{New-18}
\mathcal{X}(t) \leq \mathcal{X}_0 +\int_0^t\mathcal{V}(s)ds \leq \mathcal{X}_0 + \mathcal{V}_0 \int_0^te^{-\underline{a} s}ds \leq \mathcal{X}_0 + \frac{\mathcal{V}_0}{\underline{a}}.
\end{equation}

\vspace{0.5cm}

\noindent $\bullet$~Case B (Estimate for ${\mathcal E}$):~we use  \eqref{XVdef} and $\eqref{KCS}_3$ to find 
\begin{align*}
\begin{aligned}
\frac{d \mathcal{E}^2}{dt} &= \frac{2}{n}  \sum_{\al, \beta=1}^n  a_{\al\beta} \left(T_{\al} +\frac{1}{2}|\bu_{\al}|^2-T_{\infty}\right) \left(T_{\beta} + \frac{1}{2}|\bu_{\beta}|^2-T_{\al}-\frac{1}{2}|\bu_{\al}|^2\right) \\
&= \frac{2}{n}  \sum_{\al, \beta=1}^n  a_{\al\beta} E_{\al} \left(\Big(T_{\beta} + \frac{1}{2}|\bu_{\beta}|^2 - T_\infty \Big) - \Big( T_{\al} + \frac{1}{2}|\bu_{\al}|^2 - T_\infty \Big) \right) \\
&=  \frac{2}{n}  \sum_{\al, \beta=1}^n  a_{\al\beta} E_\alpha \cdot (E_\beta - E_\alpha) =  -\frac{1}{n}  \sum_{\al, \beta=1}^n  a_{\al\beta} |E_\beta - E_\alpha|^2.
\end{aligned}
\end{align*}
Now, we use \eqref{New-16-1} to find
\begin{align}
\begin{aligned} \label{EKCS}
\frac{d \mathcal{E}^2}{dt} &= -\frac{1}{n}  \sum_{\al=1}^n \sum_{\beta=1}^na_{\al\beta} | E_{\beta}-E_{\al}|^2  \leq  -\frac{\underline{a}}{n}  \sum_{\al=1}^n \sum_{\beta=1}^n  | E_{\beta}-E_{\al}|^2 \\
&= -\frac{\underline{a}}{n}  \sum_{\al=1}^n \sum_{\beta=1}^n  \Big( | E_{\beta}|^2 -2 E_\alpha E_\beta + |E_{\al}|^2 \Big) \\
& =  -2 \underline{a}  \sum_{\al=1}^n | E_{\alpha}|^2 =  -2 \underline{a} \mathcal{E}^2.
\end{aligned}
\end{align}
This yields the desired estimate. 

\vspace{0.5cm}

\noindent (ii) ~Suppose that the network topology $(a_{\alpha \beta})$ satisfies Type B conditions \eqref{NNN-3}.
We use the same argument employed in \eqref{VKCS} to find 
\[
\frac{d\mathcal{V}^2}{dt}  \leq - 2 \phi(t) \mathcal{V}^2 \leq 0,
\]
where $\phi(t)$ is defined in \eqref{NNN-4}.
This implies 
\[
\mathcal{V}(t) \leq \mathcal{V}_0 \quad \mbox{and} \quad  \mathcal{X}(t) \leq \mathcal{X}_0+\mathcal{V}_0t.
\]
Again, we have
\[ |\bx_{\beta}-\bx_{\al}| %\leq \left(\sum_{\al=1}^n\sum_{\beta=1}^n |\bx_{\beta}-\bx_{\al}|^2\right)^{\frac{1}{2}} 
\leq \mathcal{X}(t) \leq \mathcal{X}_0+\mathcal{V}_0t.
\]
Now we use this and \eqref{NNN-4} to get  
\begin{equation} \label{philower}
\phi(t) \geq \frac{1}{(1+|\mathcal{X}_0+\mathcal{V}_0t|^2)^{\lambda}} 
\geq \frac{1}{(\max\{1+2\mathcal{X}_0^2,2\mathcal{V}_0^2\}(1+t)^2)^{\lambda}} \geq \frac{\Lambda_0}{(1+t)^{2\lambda}},
\end{equation}
where $\Lambda_0$ is defined in $\eqref{NNN-5}_2$. \newline

\noindent Next, we combine \eqref{VexpKCS} and \eqref{philower} to find 
\begin{equation} \label{New-19}
\mathcal{V}(t) \leq \mathcal{V}_0 e^{-k \int_0^t \frac{ds}{(1+s)^{2\lambda}}}.
\end{equation}
On the other hand, one has 
\begin{equation} \label{New-20}
\int_0^t \frac{\Lambda_0}{(1+s)^{2\lambda}} ds
= \begin{cases} 
\Lambda_0 \ln(1+t), \quad & \lambda= \frac{1}{2}, \\ 
\frac{\Lambda_0}{1-2\lambda}((1+t)^{1-2\lambda}-1) , \quad & \lambda \in \Big(0,  \frac{1}{2} \Big). 
\end{cases}
\end{equation}
We combine \eqref{New-19} and \eqref{New-20} to find 
\[
\mathcal{V}(t) \leq \mathcal{V}_0 \begin{cases} e^{-\frac{\Lambda_0}{1-2\lambda}((1+t)^{1-2\lambda}-1)}, \quad & \lambda \in (0, \frac{1}{2}), \\ 
(1+t)^{-\Lambda_0}, \quad &\lambda=\frac{1}{2}.
\end{cases} 
\]
Next, we use \eqref{XV1} to get 
\begin{align*}
\mathcal{X}(t) &\leq \mathcal{X}_0 + \int_0^t \mathcal{V}(s) ds 
\leq \mathcal{X}_0 + C_{\lambda}(t)\mathcal{V}_0.
\end{align*}
On the other hand, by the slight variation of  \eqref{EKCS} and \eqref{philower}, we have
\[
\frac{d \mathcal{E}^2}{dt} \leq  -\frac{2\Lambda_0}{(1+t)^{2\lambda}} \mathcal{E}^2, \quad t > 0.
\]
This yields the desired estimate. 
\end{proof}
\begin{remark} \label{R5.2}
We provide several comments on the flocking estimates for the PB-CS and KB-CS models. 
\begin{enumerate}
\item
By Theorem \ref{T5.2} and \eqref{XVdef}, we have
\[ 
\sup_{0 \leq t < \infty} |\bx_\alpha(t)| \leq \mathcal{X}_0 + \frac{\mathcal{V}_0}{\underline{a}}, \quad |\bu_\alpha(t)| \leq  \mathcal{V}_0 e^{-\underline{a} t}, \quad  |E_\alpha(t)| \leq 
\mathcal{E}_0 e^{-\underline{a} t}. \quad t \geq 0.
\]
\item
The main difference in the flocking estimate between the two models is that the model \eqref{KCS} emerges flocking for any initial data without smallness condition, while the model \eqref{TCS} emerges flocking only in a small diffusion case (see Theorem \ref{T5.1} and Theorem \ref{T5.2}). 
\end{enumerate}
\end{remark}
As a direct corollary of Theorem \ref{T5.1} and Theorem \ref{T5.2}, we have asymptotic equivalence for the PB-CS and the KB-CS models, when the thermal and mechanical diffusions are sufficiently small:
\[  \max_{\alpha} |T_{\al}(0)-T_0| \ll 1 \quad \mbox{and} \quad  \max_{\alpha} |\bu_\alpha(0) |  \ll 1.\]
Let $(\bx_{\al}^P,\bu_{\al}^P,T_{\al}^P)$ and $(\bx_{\al}^K,\bu_{\al}^K,T_{\al}^K)$ be the solutions to \eqref{TCS} and \eqref{KCS}, respectively.  Here the superscripts $P$ and $K$ denote the initials of PB-CS and KB-CS. We define $E_{\al}^P$ and $E_{\al}^K$ to satisfy the relation \eqref{XVdef}.
\begin{corollary}\label{C5.1} 
Suppose that the network topology satisfies Type A conditions \eqref{NNN-2}, and common initial data satisfy \eqref{initial} and \eqref{New-16-2} and let $\{ (\bx_{\al}^P,\bu_{\al}^P,T_{\al}^P)\}$ and $\{ (\bx_{\al}^K,\bu_{\al}^K,T_{\al}^K) \}$ be the global smooth solutions of \eqref{TCS} and \eqref{KCS}, respectively. Then, there exists a positive constant $C>0$ such that 
\begin{equation} \label{New-21}
\begin{cases}
\displaystyle   |\bu_{\al}^P(t)-\bu_{\al}^K(t)| \leq C\varepsilon e^{-\underline{a} t}, \quad t \geq 0, \quad  \al \in [n], \\
\displaystyle   |\bx_{\al}^P(t)-\bx_{\al}^K(t)| \leq C \varepsilon, \quad  |E_{\al}^P(t)-E_{\al}^K(t)| \leq C\varepsilon e^{-\underline{a} t}.
\end{cases}
\end{equation}
\end{corollary}
\begin{proof} 
First note that the conditions \eqref{New-16-2} imply
\begin{align*}
\begin{aligned}
& \frac{1}{2} |\bu_\alpha(0)|^2 \leq \sum_{\alpha=1}^{n}  \frac{1}{2} |\bu_\alpha(0)|^2 \leq  \frac{\varepsilon^2}{16}, \quad |T_\alpha(0) - T_0| < \frac{\varepsilon}{2\sqrt{2}}, \\
&  |E_\alpha(0)| \leq \frac{1}{2}  |\bu_\alpha(0)|^2  +  |T_\alpha(0) - T_0| \leq \frac{\varepsilon^2}{16} +  \frac{\varepsilon}{2\sqrt{2}}.
\end{aligned}
\end{align*}
These yield 
\begin{equation} \label{NNN-7}
{\mathcal V}_0 =   \left(\sum_{\al=1}^n |\bu_{\al}(0)|^2 \right)^{\frac{1}{2}} \leq \frac{\varepsilon}{2\sqrt{2}}, \quad  {\mathcal E}_0 = \left(\sum_{\al=1}^n \Big|E_\alpha(0) \Big|^2 \right)^{\frac{1}{2}} \leq \sqrt{n} \Big( \frac{\varepsilon^2}{16} +  \frac{\varepsilon}{2\sqrt{2}} \Big).
\end{equation}
On the other hand, it follows from Theorem \ref{T5.1} and Theorem \ref{T5.2} that 
\begin{align} 
\begin{aligned} \label{NNN-8}
&  |\bx^P_\alpha(t) - \bx_\alpha(0) | \leq \frac{C {\mathcal V}_0}{\underline{a}}, \quad  |\bu^P_\alpha(t)  | \leq C \mathcal{V}_0 e^{-\underline{a} t}, \quad  |E^P_\alpha(t)| \leq C \mathcal{E}_0 e^{-\underline{a} t}, \\
& |\bx^K_\alpha(t) - \bx_\alpha(0)| \le \frac{C \mathcal{V}_0}{\underline{a}}, \quad |\bu^K_\alpha(t)| \leq  C \mathcal{V}_0 e^{-\underline{a} t}, \quad  |E^K_\alpha(t)| \leq C \mathcal{E}_0 e^{-\underline{a} t}. 
\end{aligned}
\end{align}
Finally, we use triangle inequality, \eqref{NNN-7} and \eqref{NNN-8} to derive \eqref{New-21}.
\end{proof}

\subsection{Intermediate dynamics} \label{sec:5.2}
In this subsection, we further study dynamic discrepancy between two models \eqref{TCS} and \eqref{KCS}. To see this, we consider two time-independent network topologies:
\begin{enumerate}
\item Uniform constant interaction weight:
\begin{align} \label{a1}
a_{\al\beta}=1 \quad \forall~\alpha, \beta \in [n].
\end{align}
\item
All-to-all symmetric interaction weight:
\begin{align} \label{a2}
a_{\alpha \beta} = a_{\beta \alpha} > 0, \quad \forall~\al,\beta \in [n].
\end{align}
\end{enumerate} 
Note that the first network topology corresponds to the special case of the second network topology. 

\subsubsection{Uniform constant interaction weight} \label{sec:5.2.1} Consider uniform constant communication weight:
\[  a_{\al\beta}=1, \quad \forall~ \al,\beta \in [n]. \] 
In this case, the KB-CS model \eqref{KCS} with \eqref{initial}  becomes a decoupled system:
\begin{align} \label{New-25-1}
\frac{d\bx_{\al}}{dt} = \bu_{\al}, \quad  \frac{d\bu_{\al}}{dt} = -\bu_{\al}, \quad  \frac{d E_{\al}}{dt} = -E_{\al} , \quad \alpha \in [n],
\end{align}
where we used the conservation laws in \eqref{conserved}. These imply
\begin{align}
\begin{aligned} \label{New-26}
& \bu_{\al}(t) = \bu_{\al}(0) e^{-t}, \quad E_{\al}(t) = E_{\al}(0) e^{-t} \quad
\bx_\al(t) = \bx_\al(0) + \bu_{\al}(0)(1-e^{-t}), 
\end{aligned}
\end{align}
By letting $t \to \infty$, one can see that all particles tend to a flocking state:
\[  \lim_{t \to \infty} (\bx_\alpha(t), \bu_\alpha(t), T_\alpha(t)) = (\bx_\al(0) + \bu_\alpha(0), 0, T_0), \quad \forall~\alpha \in [n].        \]
Therefore, the KB-CS model exhibits an asymptotic flocking behavior for any initial data. Furthermore, $|\bu_{\al}(t)|$ and $E_{\al}(t)$ are strictly decreasing for all $t \geq 0$ and 
\begin{align}
\begin{aligned} \label{New-26.5}
& \mathcal{V}_{KBCS}^2(t) = \sum_{\al=1}^n |\bu_{\al}(t)|^2 = \sum_{\al=1}^n |\bu_{\al}(0)|^2 e^{-2t}, \cr 
&\mathcal{E}_{KBCS}^2(t) =\sum_{\al=1}^n |E_\alpha(t)|^2 = \sum_{\al=1}^n |E_\alpha(0)|^2e^{-2t},
\end{aligned}
\end{align}
where $\mathcal{V}_{KBCS}$ and $\mathcal{E}_{KBCS}$ are $\mathcal{V}$ and $\mathcal{E}$ for the KB-CS model \eqref{KCS} defined in \eqref{XVdef}.
 \newline

Similarly, the PB-CS model \eqref{TCS} represents asymptotic flocking of velocity variable for any initial data. (We will show that in Lemma \ref{nflock}.) On the other hand, each $|\bu_{\al}(t)|$ and $E_{\al}(t)$ may not exhibit monotonic behaviors for some network topology and initial data. To be specific, we consider a three-particle system on the one-dimensional line:
\begin{align}
\begin{aligned} \label{New-27}
\bu_{1}(0) &= (1,0,0), \quad \bu_{2}(0) = (2,0,0), \quad \bu_{3}(0) = (-3,0,0), \cr 
T_{1}(0) &= 1, \quad T_{2}(0) = 0.1, \quad T_{3}(0) = 1.
\end{aligned}
\end{align}
Note that 
\begin{equation*}
\sum_{i=1}^{3} \bu_i(0) = 0, \quad T_0 = \frac{1}{3} \sum_{\al=1}^3 \left(T_\al(0)+\frac{1}{2} u_\al^2(0)\right) \approx 3.033. 
\end{equation*}
Then we have 
\begin{align}
\begin{aligned} \label{New-28}
\frac{d(\bu_1)_1}{dt}\bigg|_{t=0} &= \frac{T_0}{3}\sum_{\beta=1}^3 \left(\frac{\bu_{\beta}(0)}{T_\beta(0)}-\frac{\bu_1(0)}{T_1(0)}\right) \cr 
&= \frac{T_0}{2} \lw(\Big(\frac{2}{0.1}-\frac{1}{1}\Big)+\Big(\frac{-3}{1}-\frac{1}{1}\Big)\rw) = \frac{15}{2}T_0 > 0.
\end{aligned}
\end{align}
From Lemma \ref{nflock}, the velocity variable $\bu_i(t)$ will converges to $0$:
\[  \lim_{t \to \infty} \bu_i(t) = \frac{1}{3} \sum_{i=1}^{3} \bu_i(0) = 0.  \]
Due to \eqref{New-28}, $|\bu_1(t)|$ is initially increasing but it should decrease at some time in order to converge to $0$. Hence, $|\bu_1|$ is not monotonic.

\begin{lemma}\label{nflock} Suppose that network topology satisfies uniform constant communication weight \eqref{a1}, and initial data satisfy \eqref{initial}. 
Let $(\bx_{\al},\bu_{\al},T_{\al})$ be the solution of system \eqref{TCS}, then we have
\begin{align}\label{VTR}
\frac{d\mathcal{V}_{PBCS}^2}{dt} = -2 \sum_{\al=1}^{n} \frac{T_0}{T_{\al}}|\bu_{\al}|^2, \qquad \mathcal{V}_{PBCS}(t) &\leq  e^{-\frac{1}{n}t}\mathcal{V}_{PBCS}(0),
\end{align}
where $\mathcal{V}_{PBCS}$ is $\mathcal{V}$ for the PB-CS model \eqref{TCS} defined in \eqref{XVdef}. %Moreover, the equality holds only when $\bu_1=\cdots=\bu_n=0$. 
\end{lemma}
\begin{proof}
We take $\sum_{\al=1}^n \bu_{\al}$ on the second equation of \eqref{TCS}, then use the index exchange transformation $\alpha~\leftrightarrow~\beta$ to have 
\begin{align*}
\frac{d\mathcal{V}_{PBCS}^2}{dt} &= -\frac{1}{n} \sum_{\al=1}^n\sum_{\beta=1}^n \left(\frac{T_0}{T_{\al}}|\bu_{\al}|^2 - \lw(\frac{T_0}{T_{\al}}+\frac{T_0}{T_{\beta}}\rw)\bu_{\al}\cdot\bu_{\beta}+\frac{T_0}{T_{\beta}}|\bu_{\beta}|^2\right).
\end{align*}
We decompose the summation $\sum_{\al=1}^n\sum_{\beta=1}^n$ as follows:
\begin{align}\label{VT1}
\bega
\frac{d\mathcal{V}_{PBCS}^2}{dt} &= -\frac{1}{n} \sum_{\al=1}^{n-1}\sum_{\beta=1}^{n-1} \left(\frac{T_0}{T_{\al}}|\bu_{\al}|^2 - \lw(\frac{T_0}{T_{\al}}+\frac{T_0}{T_{\beta}}\rw)\bu_{\al}\cdot\bu_{\beta}+\frac{T_0}{T_{\beta}}|\bu_{\beta}|^2\right) \cr 
&-\frac{2}{n} \sum_{\al=1}^{n-1} \left(\frac{T_0}{T_{\al}}|\bu_{\al}|^2 - \lw(\frac{T_0}{T_{\al}}+\frac{T_0}{T_n}\rw)\bu_{\al}\cdot\bu_n+\frac{T_0}{T_n}|\bu_n|^2\right),
\enda
\end{align}
where we used that the term inside the large parenthesis is zero when $(\al,\beta)=(n,n)$.
The momentum conservation law $\eqref{conserved}_2$ implies
\begin{align}\label{conseru}
\bu_n = -\bu_1-\cdots -\bu_{n-1} = -\sum_{\beta=1}^{n-1}\bu_{\beta}.
\end{align}
Substituting \eqref{conseru} on the second term of second line of \eqref{VT1} yields 
\begin{align}\label{VT2}
\bega
\frac{2}{n}\sum_{\al=1}^{n-1} \lw(\frac{T_0}{T_{\al}}+\frac{T_0}{T_n}\rw)\bu_{\al}&\cdot \bu_n =-\frac{2}{n}\sum_{\al=1}^{n-1} \lw(\frac{T_0}{T_{\al}}+\frac{T_0}{T_n}\rw)\bu_{\al}\cdot\Big(\sum_{\beta=1}^{n-1}\bu_{\beta}\Big) \cr 
&= -\frac{1}{n}\sum_{\al=1}^{n-1}\sum_{\beta=1}^{n-1} \lw(\frac{T_0}{T_{\al}}+\frac{T_0}{T_{\beta}}+2\frac{T_0}{T_n}\rw) \bu_{\al}\cdot\bu_{\beta},
\enda
\end{align}
where we used the index exchange transformation $\alpha~\leftrightarrow~\beta$. Combining \eqref{VT1} and \eqref{VT2}, we get 
\begin{align}\label{VT3}
\bega
\frac{d\mathcal{V}_{PBCS}^2}{dt} &= -\frac{1}{n} \sum_{\al=1}^{n-1}\sum_{\beta=1}^{n-1} \left(\frac{T_0}{T_{\al}}|\bu_{\al}|^2 +\frac{T_0}{T_{\beta}}|\bu_{\beta}|^2\right) \cr 
&-\frac{2}{n} \sum_{\al=1}^{n-1} \left(\frac{T_0}{T_{\al}}|\bu_{\al}|^2 +\frac{T_0}{T_n}|\bu_n|^2\right) - \frac{2}{n}\frac{T_0}{T_n}\sum_{\al=1}^{n-1}\bu_{\al}\cdot \sum_{\beta=1}^{n-1}\bu_{\beta}.
\enda
\end{align}
Using \eqref{conseru} on the last term, we obtain
\begin{align}\label{VT4}
\bega
\frac{d\mathcal{V}_{PBCS}^2}{dt} &= -\lw(\frac{2n-2}{n}+\frac{2}{n}\rw) \sum_{\al=1}^{n-1} \frac{T_0}{T_{\al}}|\bu_{\al}|^2  -\frac{2}{n} \sum_{\al=1}^{n-1} \frac{T_0}{T_n}|\bu_n|^2 - \frac{2}{n}\frac{T_0}{T_n}|\bu_n|^2 \cr 
&= -2 \sum_{\al=1}^{n} \frac{T_0}{T_{\al}}|\bu_{\al}|^2.
\enda
\end{align}
By the energy conservation law $\eqref{conserved}_3$, we have $T_{\al}(t) \leq nT_0$ for $\forall \al\in[n]$, $t\geq0$. This gives 
\begin{align*}
\frac{d\mathcal{V}_{PBCS}^2}{dt} &\leq  -\frac{2}{n}\mathcal{V}_{PBCS}^2.
\end{align*}
Then Gr\"{o}nwall's inequality gives the desired result. 
\end{proof}

\subsubsection{All-to-all symmetric interaction weight} \label{sec:5.2.2} 
In this part, we assume the all-to-all symmetric interaction weight:
\begin{equation} \label{New-29}
 a_{\al\beta} = a_{\beta \alpha} >0,  \quad \forall~\al,\beta \in [n]. 
 \end{equation}
 In what follows, we discuss the non-monotonic behavior of velocity profile to the KB-CS model \eqref{KCS}. For this, we consider the four-particle system with the following initial velocity profile:
\begin{align*}
\bu_1(0) = (1,0,0), \quad \bu_2(0) = (2,0,0), \quad \bu_3(0) = (-1,0,0), \quad \bu_4(0) = (-2,0,0),
\end{align*}
and we choose the communication matrix as follows: 
\begin{align*}
(a_{\al\beta}) = \lw[\begin{array}{cccc} 0 & 10 & 1 & 1 \\ 10 & 0 & 1 & 1 \\ 1 & 1 & 0 & 1 \\ 1 & 1 & 1 & 0\end{array} \rw].
\end{align*}
Note that the interactions between the first and second particles are much stronger than other interactions.
Then, it is easy to check that  $\sum_{\al=1}^4\bu_{\al}(0)=0$ and
\begin{align}\label{KCSinc}
\begin{aligned}
\frac{d\bu_1}{dt}\bigg|_{t=0} &= \frac{1}{4}\sum_{\beta=1}^4 a_{1 \beta}\left(\bu_{\beta}-\bu_1\right) = \lw(\frac{1}{4}(10-2-3) , 0 ,0 \rw).
\end{aligned}
\end{align}
Thus, the first component of the velocity profile $\bu_1$  is in increasing mode at $t=0$. However, Theorem \ref{T5.2} illustrates that for any initial data with zero total momentum, $\mathcal{V}$ tends to zero exponentially fast. In particular, $\bu_1(t)$ converges to zero exponentially fast. Hence, the first component of $\bu_1$ should be in decreasing mode at some positive instant. In what follows, we are interested in the following question: \newline
\begin{quote}
"{\it Under the symmetry condition \eqref{New-29}, what can we say about the difference between \eqref{TCS} and \eqref{KCS}? }"
\end{quote}
\vspace{0.2cm}
We will observe different intermediate dynamics for velocity fluctuation $\mathcal{V}$ and energy fluctuation $\mathcal{E}$ for each model. For simplicity, we set ${\mathcal V}_{PBCS}, {\mathcal E}_{PBCS}$ and ${\mathcal V}_{KBCS}, {\mathcal E}_{KBCS}$ to denote $\mathcal{V}, \mathcal{E}$ for \eqref{TCS} and \eqref{KCS}, respectively:

\begin{proposition} \label{P5.1} Let $\{ (\bx_{\al}^K,\bu_{\al}^K,T_{\al}^K) \}$ be a global smooth solution of \eqref{KCS}. Then, ${\mathcal V}_{KBCS}, {\mathcal E}_{KBCS}$ are monotonically decreasing:
\begin{equation} \label{NNN-10}
\frac{d\mathcal{V}_{KBCS}^2}{dt} \leq 0 ,\qquad \frac{d \mathcal{E}_{KBCS}^2}{dt} \leq 0, \quad \forall~t > 0,
\end{equation}
where  equalities hold if and only if $\bu_{\al}=\bu_{\beta}$ and $E_{\al}=E_{\beta}$, for all $\al,\beta \in [n]$, respectively.
\end{proposition}
\begin{proof}
Since the desired estimate can be obtained directly from the estimate \eqref{VKCS} and \eqref{EKCS}, we omit its details. 
%\begin{align}
%\begin{aligned}
%\frac{d\mathcal{V}_{KBCS}^2}{dt} = -\frac{1}{n} \sum_{\al=1}^n\sum_{\beta=1}^n a_{\al\beta} |\bu_{\beta}-\bu_{\al}|^2  
%\end{aligned}
%\end{align}
%\begin{align}
%\begin{aligned}
%\frac{d \mathcal{E}_{KBCS}^2}{dt} &= -\frac{1}{n}  \sum_{\al=1}^n \sum_{\beta=1}^na_{\al\beta} | E_{\beta}-E_{\al}|^2 
%\end{aligned}
%\end{align}
\end{proof}
\begin{remark} \label{R5.3}
Note that the estimates \eqref{NNN-10} hold for any initial data satisfying the constraints \eqref{initial}. 
\end{remark}

\vspace{0.5cm}

In the next two propositions, we show that velocity and energy fluctuations to the model \eqref{TCS} can increase initially for some well-prepared initial data.
\begin{proposition}\label{P5.2} Suppose that initial data and network topology satisfy  \eqref{initial} and 
\begin{align}
\begin{aligned} \label{Vini}
& 0<T_2^P(0)<T_1^P(0), \quad \frac{T_1^P(0)+T_2^P(0)}{2T_1^P(0)}\bu_1^P(0)=\bu_2^P(0) \neq 0, \\
& a_{12} >\frac{2T_1^P(0)T_2^P(0)}{|T_1^P(0)-T_2^P(0)|^2} \sum_{\substack{1\leq\al,\beta\leq n \\ (\al,\beta)\neq(1,2),(2,1)}} a_{\alpha \beta}\bigg(\frac{|\bu_{\al}^P(0)|^2}{T_{\al}^P(0)} \cr 
&\hspace{1cm} + \bigg(\frac{1}{T_{\al}^P(0)}+\frac{1}{T_{\beta}^P(0)}\bigg)\bu_{\al}^P(0) \cdot\bu_{\beta}^P(0)+\frac{|\bu_{\beta}^P(0)|^2}{T_{\beta}^P(0)}\bigg),
\end{aligned}
\end{align}
and let $\{ (\bx_{\al}^P,\bu_{\al}^P,T_{\al}^P)\}$ be a global smooth solution of \eqref{TCS}. Then we have 
\begin{align*}
\begin{aligned}
\frac{d\mathcal{V}_{PBCS}^2}{dt}\Big|_{t=0}>0.
\end{aligned}
\end{align*}
\end{proposition}
\begin{proof}
We take an inner product between $\bu_{\al}^P$ and $\eqref{TCS}_2$ to find
\begin{equation} \label{New-31}
\bega
\frac{d\mathcal{V}_{PBCS}^2}{dt}  &= \frac{d}{dt} \sum_{\al=1}^n |\bu^P_{\al}|^2 = \frac{2T_0}{n} \sum_{\al=1}^n\sum_{\beta=1}^n a_{\alpha \beta}\left(\frac{\bu_{\beta}^P}{T_{\beta}^P}-\frac{\bu_{\al}^P}{T_{\al}^P}\right)\cdot \bu_{\al}^P \cr 
&=\frac{2T_0}{n} \sum_{\al=1}^n\sum_{\beta=1}^n a_{\alpha \beta}\left(-\frac{|\bu_{\al}^P|^2}{T_{\al}^P}+\frac{\bu_{\al}^P\cdot \bu_{\beta}^P}{T_{\beta}^P} \right).
\enda
\end{equation}
We use the index exchange transformation $\al\leftrightarrow \beta$ to see 
\begin{align}\label{VTCS2}
\bega
\frac{d\mathcal{V}_{PBCS}^2}{dt} &= \frac{T_0}{n} \sum_{\al=1}^n\sum_{\beta=1}^n a_{\alpha \beta}\underbrace{\left(-\frac{|\bu_{\al}^P|^2}{T_{\al}^P} + \lw(\frac{1}{T_{\al}^P}+\frac{1}{T_{\beta}^P}\rw)\bu_{\al}^P \cdot\bu_{\beta}^P-\frac{|\bu_{\beta}^P|^2}{T_{\beta}^P}\right)}_{={\mathcal Q}(\bu_\alpha, \bu_\beta)}.
\enda
\end{align}
If $\alpha$ and $\beta$-th particles satisfy 
\[ \bu_{\al} \cdot \bu_{\beta} = |\bu_{\al}||\bu_{\beta}|\cos\theta > 0, \]
for some~~$\theta \in [-\frac{\pi}{2},\frac{\pi}{2}]$ with different temperatures, then they contribute a positive effect on $\frac{d\mathcal{V}^2}{dt}$:
\[ \lw(\frac{T_0}{T_{\al}}+\frac{T_0}{T_{\beta}}\rw)\bu_{\al} \cdot\bu_{\beta} > 0. \]
If $\bu_{\al}=k\bu_{\beta}$ with $k>0$, then we have the following quadratic form:
\begin{align*}
\begin{aligned}
{\mathcal Q}(\bu_\alpha, \bu_\beta) &= -\frac{1}{T_{\al}}|\bu_{\al}|^2 + \Big( \frac{1}{T_{\al}}+\frac{1}{T_{\beta}} \Big) |\bu_{\al}| \cdot |\bu_{\beta}| - \frac{1}{T_{\beta}}|\bu_{\beta}|^2 \\
&= -\frac{1}{T_\alpha} ( |\bu_\alpha| - |\bu_\beta|) \Big (|\bu_\alpha| - \frac{T_\alpha}{T_\beta} |\bu_\beta| \Big ).
\end{aligned}
\end{align*}
Thus, it follows from the assumption on initial data \eqref{Vini} that 
\begin{align}\label{Va12}
{\mathcal Q}(\bu_1(0), \bu_2(0)) =\frac{|T_1^P(0)-T_2^P(0)|^2}{4T_1^P(0)T_2^P(0)} > 0.
\end{align}
We combine \eqref{VTCS2} and \eqref{Va12} to get 
\begin{align}\label{New-31b}
\bega
\frac{d\mathcal{V}_{PBCS}^2}{dt}\Big|_{t=0} &= \frac{T_0}{n}a_{12}\frac{|T_1^P(0)-T_2^P(0)|^2}{2T_1^P(0)T_2^P(0)} - \frac{T_0}{n} \sum_{\substack{1\leq\al,\beta\leq n \\ (\al,\beta)\neq(1,2),(2,1)}} a_{\alpha \beta}\bigg(\frac{|\bu_{\al}^P(0)|^2}{T_{\al}^P(0)} \cr 
&\quad + \lw(\frac{1}{T_{\al}^P(0)}+\frac{1}{T_{\beta}^P(0)}\rw)\bu_{\al}^P(0) \cdot\bu_{\beta}^P(0)+\frac{|\bu_{\beta}^P(0)|^2}{T_{\beta}^P(0)}\bigg).
\enda
\end{align}
Then the condition \eqref{Vini} gives the desired result. 
\end{proof}
\begin{remark} We comment on the result of Proposition \ref{P5.2} as follows.
\begin{enumerate}
\item
The assumptions of Proposition \ref{P5.2} is non-empty. We set 
\begin{align*}
\bega
\bu_{1}(0) &= (4,0,0), \quad \bu_{2}(0) = (3,0,0), \quad \bu_{3}(0) = (-7,0,0), \cr 
T_{1}(0) &= 2, \quad T_{2}(0) = 1, \quad T_{3}(0) = 1.
\enda
\end{align*}
Then it satisfies \eqref{initial} and \eqref{Vini} with $T_0=\frac{41}{3}$. 
%\[
%T_0 = \frac{1}{3} \sum_{\al=1}^3 \left(T_\al(0)+\frac{1}{2} u_\al^2(0)\right) %=\frac{1}{3}(4+\frac{1}{2}74) 
%= \frac{41}{3}, 
%\]
Moreover, the following communication function
\[ 
(a_{\al\beta})= \lw[\begin{array}{ccc} 0 & 200 & 1 \\ 
200 & 0 & 1 \\ 1 & 1 & 0 \end{array} \rw]
\]
satisfies $\eqref{Vini}_2$. 
More definitively, it follows from \eqref{VTCS2} that 
\begin{align*}
\bega
\frac{d\mathcal{V}_{PBCS}^2}{dt}\Big|_{t=0} 
%&= \frac{T_0}{n} \sum_{\al=1}^n\sum_{\beta=1}^n a_{\alpha \beta}\left(-\frac{1}{T_{\al}}|\bu_{\al}|^2 + \lw(\frac{1}{T_{\al}}+\frac{1}{T_{\beta}}\rw)\bu_{\al} \cdot\bu_{\beta}-\frac{1}{T_{\beta}}|\bu_{\beta}|^2\right) \cr
%&= \frac{2T_0}{3} a_{12}\left(-\frac{1}{2}4^2+(\frac{1}{2}+1)12-3^2\right) \cr 
%&+ \frac{2T_0}{3} a_{13}\left(-\frac{1}{2}4^2 + (\frac{1}{2}+1)(-28)-7^2\right) \cr 
%&+ \frac{2T_0}{3} a_{23}\left(-3^2 + 2(-21)-7^2\right) \cr 
&= \frac{2T_0}{3}\lw(a_{12}-99a_{13}-100a_{23} \rw)>0.
\enda
\end{align*}
\item
From \eqref{VTCS2}, we can see that $\mathcal{V}_{PBCS}$ always decreases for a two-particle system with any initial data satisfying \eqref{initial}. By the conservation $\eqref{conserved}_2$ on the center of momentum frame, one has 
\[ \bu_\al + \bu_\beta = 0, \quad \bu_{\al}\cdot\bu_{\beta}=-|\bu_{\al}||\bu_{\beta}|. \]
This implies  
\begin{align*}
\bega
\frac{d\mathcal{V}_{PBCS}^2}{dt} &= -\frac{1}{n} \sum_{\al=1}^2\sum_{\beta=1}^2 a_{\alpha \beta}\left(\frac{T_0}{T_{\al}}|\bu_{\al}|^2 + \bigg|\frac{T_0}{T_{\al}}+\frac{T_0}{T_{\beta}}\bigg||\bu_{\al}||\bu_{\beta}|+\frac{T_0}{T_{\beta}}|\bu_{\beta}|^2\right) \cr 
&\leq -\frac{1}{n} \sum_{\al=1}^2\sum_{\beta=1}^2 a_{\alpha \beta}\left(\frac{T_0}{T_{\al}}|\bu_{\al}|^2 + 2\sqrt{\frac{T_0^2}{T_{\al}}T_{\beta}}|\bu_{\al}||\bu_{\beta}|+\frac{T_0}{T_{\beta}}|\bu_{\beta}|^2\right) \cr
&\leq -\frac{1}{n} \sum_{\al=1}^2\sum_{\beta=1}^2 a_{\alpha \beta}\left(\sqrt{\frac{T_0}{T_{\al}}}|\bu_{\al}|+\sqrt{\frac{T_0}{T_{\beta}}}|\bu_{\beta}|\right)^2 \leq 0,
\enda
\end{align*}
where we used the inequality $a+b\geq 2\sqrt{ab}$.
\end{enumerate}
\end{remark}
\begin{proposition}\label{P5.3} 
Suppose that initial data and network topology satisfy \eqref{initial} and 
\begin{align}
\begin{aligned} \label{Eacond}
&\sum_{(\al, \beta)\in \mathrm{E}_+} a_{\alpha \beta}\frac{T_0^2}{T_{\al}^P(0)T_{\beta}^P(0)} \left(E_{\al}^P(0)-E_{\beta}^P(0) \right)\left(T_{\beta}^P(0) - T_{\al}^P(0) \right) \cr 
&\hspace{0.8cm} > \sum_{(\al, \beta)\in \mathrm{E}_-} a_{\alpha \beta}\frac{T_0^2}{T_{\al}^P(0)T_{\beta}^P(0)} \left(E_{\beta}^P(0)-E_{\al}^P(0) \right)\left(T_{\beta}^P(0) - T_{\al}^P(0) \right),
\end{aligned}
\end{align}
where the index sets $\mathbf{E}_+$ and $\mathbf{E}_-$ are defined as follows: 
\begin{align*}
\begin{aligned}
\mathbf{E}_+ &:= \lw\{(\al,\beta)\in[n]\times[n] ~|~ (E_{\al}(0)-E_{\beta}(0))(T_{\beta}(0)-T_{\al}(0))>0\rw\},\cr\mathbf{E}_- &:= ([n]\times[n]) / \mathbf{E}_+,
\end{aligned}
\end{align*}
and let $\{ (\bx_{\al}^P,\bu_{\al}^P,T_{\al}^P)\}$ be a global smooth solution of \eqref{TCS}. Then we have 
\begin{align}
\begin{aligned}
\frac{d\mathcal{E}_{PBCS}^2}{dt}\Big|_{t=0}>0.
\end{aligned}
\end{align}
\end{proposition}
\begin{proof}
We multiply $E_{\al}^P$ on $\eqref{TCS}_3$ to get
\begin{align}
\begin{aligned} \label{ETCS}
\frac{d}{dt}  \mathcal{E}_{PBCS}^2  &= 2  \sum_{\al=1}^n E_\alpha^P \cdot  \frac{d E_\alpha^P}{dt} = \frac{2}{n}  \sum_{\al=1}^n E_\alpha^P \sum_{\beta=1}^n a_{\alpha \beta}\left(\frac{T_0^2}{T_{\al}^P} -\frac{T_0^2}{T_{\beta}^P} \right) \cr 
&=  \frac{1}{n} \sum_{\al, \beta=1}^n a_{\alpha \beta} \left(E_{\al}^P-E_{\beta}^P \right) \left(\frac{T_0^2}{T_{\al}^P} -\frac{T_0^2}{T_{\beta}^P} \right)  \cr 
&=  \frac{1}{n} \sum_{\al, \beta=1}^n a_{\alpha \beta} \frac{T_0^2}{T_{\al}^PT_{\beta}^P} \left(E_{\al}^P-E_{\beta}^P \right) \left(T_{\beta}^P-T_{\al}^P\right).
\end{aligned}
\end{align}
Note that the component $(\al,\beta)\in\mathbf{E}_+$ and $(\al,\beta)\in\mathbf{E}_-$ gives positive effect and negative effect on $\frac{d}{dt}  \mathcal{E}_{PBCS}^2$, respectively. Thus, we can rewrite \eqref{ETCS} as
\begin{align*}
\begin{aligned} 
\frac{d}{dt} \mathcal{E}_{PBCS}^2\Big|_{t=0}  &= \sum_{(\al, \beta)\in \mathrm{E}_+} a_{\alpha \beta}\frac{T_0^2}{T_{\al}^P(0)T_{\beta}^P(0)} \left(E_{\al}^P(0)-E_{\beta}^P(0) \right)\left(T_{\beta}^P(0) - T_{\al}^P(0) \right) \cr 
& - \sum_{(\al, \beta)\in \mathrm{E}_-} a_{\alpha \beta}\frac{T_0^2}{T_{\al}^P(0)T_{\beta}^P(0)} \left(E_{\beta}^P(0)-E_{\al}^P(0) \right)\left(T_{\beta}^P(0) - T_{\al}^P(0) \right).
\end{aligned}
\end{align*}
Finally, we use \eqref{Eacond} to obtain the desired result. 
\end{proof}
\begin{remark}
\begin{enumerate}
\item
For some initial data and communication weight, assumptions in Proposition \ref{P5.3} can be satisfied. To see this, we set 
\begin{align*}
\bega
\bu_{1}(0) &= (1,0,0), \quad \bu_{2}(0) = (-2,0,0), \quad \bu_{3}(0) = (1,0,0), \cr 
T_{1}(0) &= 2, \quad T_{2}(0) = 1, \quad T_{3}(0) = 1,  \quad (a_{\al\beta}) = \lw[\begin{array}{ccc} 0 & 3 & 1 \\ 
3 & 0 & 1 \\ 1 & 1 & 0 \end{array} \rw]. 
\enda
\end{align*}
Then they satisfy  \eqref{initial},  \eqref{Eacond}  and 
\[ E_{1}(0) = \frac{5}{2}, \quad E_{2}(0) = 3, \quad E_{3}(0) = \frac{3}{2}, \quad T_0 =\frac{7}{3}. \] 
%\[
%T_0 = \frac{1}{3} \sum_{\al=1}^3 \left(T_\al(0)+\frac{1}{2} u_\al^2(0)\right) %=\frac{1}{3}(4+\frac{1}{2}6) 
%= \frac{7}{3}, 
%\]
Thus we can see that 
\[ (1,2)\in \mathbf{E}_+ \quad \mbox{and} \quad (1,3),(2,3) \in \mathbf{E}_-, \] 
and it follows from \eqref{ETCS} that  
\begin{align*}
\begin{aligned}
\frac{d}{dt} \mathcal{E}_{PBCS}^2\Big|_{t=0}  %&=  \frac{1}{3} \sum_{\al, \beta=1}^3 a_{\alpha \beta}\frac{T_0^2}{T_{\al}^P(0)T_{\beta}^P(0)} \left(E_{\al}^P(0)-E_{\beta}^P(0) \right)\left(T_{\beta}^P(0) - T_{\al}^P(0) \right) \cr 
%&= \frac{2T_0^2}{3} a_{12}\frac{1}{2} \left(-\frac{1}{2} \right)\left(-1\right) \cr 
%&+ \frac{2T_0^2}{3} a_{13}\frac{1}{2} \left(1 \right)\left(-1\right) \cr 
%&+ \frac{2T_0^2}{3} a_{23}\frac{1}{1} \left(-\frac{1}{2} \right)\left(0\right) \cr 
&= \frac{2T_0^2}{3}\lw(\frac{1}{4}a_{12}-\frac{1}{2}a_{13}\rw) >0.
\end{aligned}
\end{align*}
\item
In Theorem \ref{T5.1}, emergent of flocking on the PB-CS model \eqref{TCS} is obtained only for small diffusion.
If the solution of system \eqref{TCS} will converges to $\bu_{\al}(t) \rightarrow 0$ and $E_{\al}(t) \rightarrow 0$ when $t\rightarrow \infty$, for all $\al\in[n]$, then we can have an upper bound for $\mathcal{V}$ and $\mathcal{E}$ which possibly increase initially (Proposition \ref{P5.2} and Proposition \ref{P5.3}),  then they should decrease after some time. 
\end{enumerate}
\end{remark}

\vspace{0.5cm}

%We expect the model \eqref{TCS} with symmetric condition \eqref{New-29} become asymptotic flocking because the space independent communication $a_{\al\beta}$ gives the same intensity of interaction as the distance increases. 

\subsection{Deviation estimate between PB-CS and KB-CS model} \label{sec:5.3}
In this subsection, we estimate the deviation between two models \eqref{TCS} and \eqref{KCS} for the space, velocity, and energy variables.

\begin{theorem}\label{smdiff}
Suppose that network topology satisfies Type A condition \eqref{NNN-2}. Let $\{ (\bx_{\al}^P,\bu_{\al}^P,T_{\al}^P)\}$ and $\{ (\bx_{\al}^K,\bu_{\al}^K,T_{\al}^K) \}$ be the $C^1$-solutions of \eqref{TCS} and \eqref{KCS} for initial data $\{(\bx_{\al}^P(0),\bu_{\al}^P(0),T_{\al}^P(0))\}$ and $\{ (\bx_{\al}^K(0),\bu_{\al}^K(0),T_{\al}^K(0)) \}$, respectively. If the initial data $\{(\bx_{\al}^P(0),\bu_{\al}^P(0),T_{\al}^P(0))\}$ satisfies \eqref{initial} and \eqref{New-16-2}, then for sufficiently small $\e$, there exists a positive constant $C>0$ such that 
\begin{align}\label{}
\begin{split}
\begin{cases} &|\bx_{\al}^P(t)-\bx_{\al}^K(t)| \leq |\bx_{\al}^P(0)-\bx_{\al}^K(0)| + \frac{1}{\tilde{a}_{\al}}|(\bu_{\al}^P(0)-\bu_{\al}^K(0))| +C\e, \\
&|\bu_{\al}^P(t)-\bu_{\al}^K(t)| \leq e^{-\tilde{a}_{\al}t}|(\bu_{\al}^P(0)-\bu_{\al}^K(0))| + C\varepsilon e^{-\frac{1}{2}\underline{a} t}, \\
&|E_{\al}^P(t)-E_{\al}^K(t)|\leq  e^{-\tilde{a}_{\al}t}|E_{\al}^P(0)-E_{\al}^K(0)| + C\varepsilon e^{-\frac{1}{2}\underline{a} t}, \end{cases}
\end{split}
\end{align}
for all $\al \in [n]$ and $t \geq 0$. Here $\tilde{a}_{\al}$ is defined by
\begin{align}\label{tildea}
\tilde{a}_{\al}:= \frac{1}{n}\sum_{\beta=1}^{n}a_{\al\beta}.
\end{align}
\end{theorem}
The proof of Theorem \eqref{smdiff} is demonstrated by the following auxiliary lemmas.

\begin{lemma} Let $\{ (\bx_{\al}^P,\bu_{\al}^P,T_{\al}^P)\}$ and $\{ (\bx_{\al}^K,\bu_{\al}^K,T_{\al}^K) \}$ be the $C^1$-solutions of \eqref{TCS} and \eqref{KCS}, respectively. Then, we have
\begin{align}
\frac{d}{dt}(\bx_{\al}^P-\bx_{\al}^K) &= \bu_{\al}^P-\bu_{\al}^K, \label{x-x} \\ 
\frac{d}{dt}(\bu_{\al}^P-\bu_{\al}^K) &= \frac{1}{n}\sum_{\beta=1}^{n}a_{\al\beta}\left( (\bu_{\beta}^P-\bu_{\beta}^K)-(\bu_{\al}^P-\bu_{\al}^K)\right)\label{u-u} \\ 
&+ \frac{1}{n}\sum_{\beta=1}^n  a_{\al\beta}  \bigg(\Big(\frac{T_0}{T_{\beta}^P}-1\Big)\bu_{\beta}^P-\Big(\frac{T_0}{T_{\al}^P}-1\Big)\bu_{\al}^P\bigg), \cr
\frac{d}{dt}(E_{\al}^P-E_{\al}^K) &= \frac{1}{n}\sum_{\beta=1}^n a_{\alpha \beta}  \Big (\frac{T_0^2}{T^P_\al} -\frac{T_0^2}{T^P_\beta} -E_{\beta}^K+ E_{\al}^K \Big ) \label{E-E}.
\end{align}
\end{lemma}
\begin{proof}
Subtracting two systems \eqref{TCS} and \eqref{KCS} directly gives the result. 
\end{proof}

\begin{lemma}\label{u-ulem}
 Under the same assumption of Theorem \ref{smdiff}, we have 
\begin{align}\label{u-uR}
\begin{split}
&|\bu_{\al}^P(t)-\bu_{\al}^K(t)| \cr 
&\leq \begin{cases} e^{-\tilde{a}_{\al}t}|(\bu_{\al}^P(0)-\bu_{\al}^K(0))| + C\varepsilon \frac{\tilde{a}_{\al}}{\tilde{a}_{\al} -\underline{a}} e^{-\underline{a} t}, \quad &\mbox{for} \quad \underline{a}< \tilde{a}_{\al}, \\
e^{-\tilde{a}_{\al}t}|(\bu_{\al}^P(0)-\bu_{\al}^K(0))| + C\varepsilon \underline{a}te^{-\underline{a}t}, \quad &\mbox{for} \quad \underline{a}= \tilde{a}_{\al}.\end{cases}
\end{split}
\end{align}
\end{lemma}
\begin{proof}
We integrate \eqref{u-u} from $0$ to $t$ to get 
\begin{align}
\begin{aligned} \label{milduu}
&\bu_{\al}^P(t)-\bu_{\al}^K(t) = e^{-\tilde{a}_{\al}t}(\bu_{\al}^P(0)-\bu_{\al}^K(0)) \\
& \hspace{0.2cm}  +  \int_0^t e^{-\tilde{a}_{\al}(t-s)} \bigg[ \frac{1}{n}\sum_{\beta=1}^{n}a_{\al\beta}(\bu_{\beta}^P(s)-\bu_{\beta}^K(s)) \cr 
& \hspace{0.2cm} + \frac{1}{n}\sum_{\beta=1}^n  a_{\al\beta}  \bigg(\Big(\frac{T_0}{T_{\beta}^P(s)}-1\Big)\bu_{\beta}^P(s)-\Big(\frac{T_0}{T_{\al}^P(s)}-1\Big)\bu_{\al}^P(s)\bigg) \bigg] ds.
\end{aligned}
\end{align}
For the third line of \eqref{milduu}, we consider the series expansion 
\[ \frac{1}{1+x} = \sum_{k\geq 0}(-x)^k, \quad \mbox{for $|x|<1$}. \]
By \eqref{New-16-2.5} in Theorem \ref{T5.1}, the temperature $T_{\al}^P(t)$ satisfies 
\begin{equation} \label{Tsmall}
\bega
\frac{|T_{\al}^P(t)-T_0|}{T_0}<\varepsilon \leq 1, \quad  t \geq 0, \quad \forall \al.
\enda
\end{equation}
Thus, we have 
\begin{align}\label{expansion}
\frac{T_0}{T_{\al}^P} = \frac{1}{1+\frac{T_{\al}^P-T_0}{T_0}} = \sum_{k\geq 0}\lw(-\frac{T_{\al}^P-T_0}{T_0}\rw)^k, \quad \bigg|\frac{T_0}{T_{\al}^P}-1\bigg| \leq \sum_{k\geq 1} \e^k.
\end{align}
We use \eqref{expansion} on \eqref{milduu} to find 
\begin{align} 
\begin{aligned} \label{milduu2}
|\bu_{\al}^P(t)&-\bu_{\al}^K(t)| \leq e^{-\tilde{a}_{\al}t}|\bu_{\al}^P(0)-\bu_{\al}^K(0)| \cr 
&+  \int_0^t e^{-\tilde{a}_{\al}(t-s)} \bigg[ \frac{1}{n}\sum_{\beta=1}^{n}a_{\al\beta}\lw(|\bu_{\beta}^P(s)|+|\bu_{\beta}^K(s)|\rw)  \bigg] ds \cr 
&+  \int_0^t e^{-\tilde{a}_{\al}(t-s)} \bigg[ \frac{1}{n}\sum_{\beta=1}^{n}a_{\al\beta}\sum_{k\geq 1} \e^k \lw(|\bu_{\beta}^P(s)|+|\bu_{\al}^P(s)|\rw)  \bigg] ds. 
\end{aligned}
\end{align}
From \eqref{New-16-2.5} in Theorem \ref{T5.1}, we also have
\begin{equation}\label{vsmall}
\bega
&|\bu_{\al}^P(t)| \leq \mathcal{V}(t)\leq C\mathcal{V}(0) e^{-\underline{a} t} \leq C \varepsilon e^{-\underline{a} t}, \quad t \geq 0, \quad \forall \al.
\enda
\end{equation}
We apply \eqref{Tsmall} and \eqref{vsmall} to  \eqref{milduu2} to get 
\begin{align}\label{uuconv}
\begin{split}
&|\bu_{\al}^P(t)-\bu_{\al}^K(t)| \leq e^{-\tilde{a}_{\al}t}|(\bu_{\al}^P(0)-\bu_{\al}^K(0))| \\
&\hspace{0.5cm} + C\e\Big(1+\sum_{k\geq 1}\varepsilon^k\Big) \int_0^t e^{-\tilde{a}_{\al}(t-s)} \tilde{a}_{\al} e^{-\underline{a} s}ds.
\end{split}
\end{align}
By explicit computations, we have 
\begin{align}\label{ee}
\begin{split}
\int_0^t e^{-\tilde{a}_{\al}(t-s)} \tilde{a}_{\al} e^{-\underline{a} s}ds \leq \begin{cases}\frac{\tilde{a}_{\al}}{\tilde{a}_{\al} -\underline{a}} e^{-\underline{a} t}, \quad &\mbox{for} \quad \underline{a}< \tilde{a}_{\al}, \\ 
\underline{a}te^{-\underline{a}t} , \quad &\mbox{for} \quad \underline{a}= \tilde{a}_{\al}. \end{cases}
\end{split}
\end{align}
Now, we combine \eqref{uuconv} and \eqref{ee} to get
\begin{align}\label{}
\begin{split}
&|\bu_{\al}^P(t)-\bu_{\al}^K(t)| \cr 
&\leq \begin{cases} e^{-\tilde{a}_{\al}t}|(\bu_{\al}^P(0)-\bu_{\al}^K(0))| + C\varepsilon \frac{\tilde{a}_{\al}}{\tilde{a}_{\al} -\underline{a}} e^{-\underline{a} t}, \quad &\mbox{for} \quad \underline{a}< \tilde{a}_{\al}, \\
e^{-\tilde{a}_{\al}t}|(\bu_{\al}^P(0)-\bu_{\al}^K(0))| + C\varepsilon \underline{a}te^{-\underline{a}t}, \quad &\mbox{for} \quad \underline{a}= \tilde{a}_{\al},\end{cases}
\end{split}
\end{align}
where we used $\underline{a}\leq \tilde{a}_{\al}$ by definition \eqref{NNN-2} and \eqref{tildea}.
\end{proof}

\begin{lemma}\label{x-xlem} 
Under the same assumption of Theorem \ref{smdiff}, we have 
\begin{align*} 
\begin{aligned}
|\bx_{\al}^P(t)-\bx_{\al}^K(t)| &\leq |\bx_{\al}^P(0)-\bx_{\al}^K(0)| + \frac{1}{\tilde{a}_{\al}}(1-e^{-\tilde{a}_{\al}t})|(\bu_{\al}^P(0)-\bu_{\al}^K(0))| \cr 
&+\begin{cases} C\varepsilon \frac{\tilde{a}_{\al}}{\tilde{a}_{\al} -\underline{a}} \frac{1}{\underline{a}}(1-e^{-\underline{a}t}), \quad &\mbox{for} \quad \underline{a}< \tilde{a}_{\al}, \\ C\varepsilon \frac{1}{\underline{a}}(1-e^{-\underline{a}t}(\underline{a}t+1)), \quad &\mbox{for} \quad \underline{a}= \tilde{a}_{\al}. \end{cases}
\end{aligned}
\end{align*}
\end{lemma}
\begin{proof}
We integrate \eqref{x-x} with respect to time to find
\begin{align}
\begin{aligned} \label{New-24}
\bx_{\al}^P(t)-\bx_{\al}^K(t) &=  \bx_{\al}^P(0)-\bx_{\al}^K(0) +  \int_0^t (\bu_{\al}^P(s)-\bu_{\al}^K(s)) ds.
\end{aligned}
\end{align}
Now, we substitute \eqref{u-uR} into \eqref{New-24} and use direct computation to get the desired estimate. 
%\begin{align} 
%\begin{aligned} \label{}
%& |\bx_{\al}^P(t)-\bx_{\al}^K(t)| \\
%& \hspace{1cm} \leq |\bx_{\al}^P(0)-\bx_{\al}^K(0)| + \frac{1}{\tilde{a}_{\al}}(1-e^{-\tilde{a}_{\al}t})|(\bu_{\al}^P(0)-\bu_{\al}^K(0))| + C\varepsilon \frac{\tilde{a}_{\al}}{\tilde{a}_{\al} -\underline{a}} \frac{1}{\underline{a}}(1-e^{-\underline{a}t}), \quad &\mbox{for} \quad \underline{a}< \tilde{a}_{\al}
%\end{aligned}
%\end{align}
%\begin{align} 
%\begin{aligned} \label{}
%& |\bx_{\al}^P(t)-\bx_{\al}^K(t)| \\
%& \hspace{1cm} \leq |\bx_{\al}^P(0)-\bx_{\al}^K(0)| + \frac{1}{\tilde{a}_{\al}}(1-e^{-\tilde{a}_{\al}t})|(\bu_{\al}^P(0)-\bu_{\al}^K(0))| + C\varepsilon \frac{1}{\underline{a}}(1-e^{-\underline{a}t}(\underline{a}t+1)), \quad &\mbox{for} \quad \underline{a}= \tilde{a}_{\al}.
%\end{aligned}
%\end{align}
\end{proof}

\begin{lemma}\label{E-Elem} 
Under the same assumption of Theorem \ref{smdiff}, we have 
\begin{align}\label{E-ER}
\begin{split}
&|E_{\al}^P(t)-E_{\al}^K(t)| \cr 
& \hspace{0.5cm} \leq \begin{cases} e^{-\tilde{a}_{\al}t}|E_{\al}^P(0)-E_{\al}^K(0)| + C\varepsilon \frac{\tilde{a}_{\al}}{\tilde{a}_{\al} -\underline{a}} e^{-\underline{a} t}, \quad &\mbox{for} \quad \underline{a}< \tilde{a}_{\al}, \\
e^{-\tilde{a}_{\al}t}|E_{\al}^P(0)-E_{\al}^K(0)| + C\varepsilon \underline{a}te^{-\underline{a}t}, \quad &\mbox{for} \quad \underline{a}= \tilde{a}_{\al}.\end{cases}
\end{split}
\end{align}
\end{lemma}
\begin{proof}
We add and subtract $\frac{1}{n}\sum_{\beta=1}^{n}a_{\al\beta} (E_{\al}^P+E_{\beta}^P)$ to the right-side of \eqref{E-E} to find 
\begin{align}
\begin{aligned} \label{New-24-1}
&\frac{d}{dt}(E_{\al}^P-E_{\al}^K) \\
& \hspace{0.2cm} = -\frac{1}{n}\sum_{\beta=1}^{n}a_{\al\beta} (E_{\al}^P-E_{\al}^K) \\
&  \hspace{0.2cm}  = \frac{1}{n}\sum_{\beta=1}^{n}a_{\al\beta} (E_{\beta}^P-E_{\beta}^K) + \frac{1}{n}\sum_{\al=1}^n \sum_{\beta=1}^n a_{\alpha \beta}\bigg(E_{\al}^P+\frac{T_0^2}{T_{\al}^P} - E_{\beta}^P-\frac{T_0^2}{T_{\beta}^P}\bigg).
\end{aligned}
\end{align}
We integrate \eqref{New-24-1} with respect to $t$ to find 
\begin{align}
\begin{aligned} \label{mildEE}
&E_{\al}^P(t)-E_{\al}^K(t) \\
& \hspace{0.4cm} = e^{-\tilde{a}_{\al}t}(E_{\al}^P(0)-E_{\al}^K(0)) \\
& \hspace{0.5cm} +  \int_0^t e^{-\tilde{a}_{\al}(t-s)} \bigg[ \frac{1}{n}\sum_{\beta=1}^{n}a_{\al\beta}(E_{\beta}^P(s)-E_{\beta}^K(s)) \cr 
& \hspace{0.5cm} + \frac{1}{n}\sum_{\alpha, \beta=1}^n  a_{\al\beta}  \bigg(E_{\al}^P(s)+\frac{T_0^2}{T_{\al}^P(s)}-T_0 - E_{\beta}^P(s)-\frac{T_0^2}{T_{\beta}^P(s)}+T_0\bigg) \bigg] ds,
\end{aligned}
\end{align}
where we subtracted and added $T_0$ on the third line. 
By Theorem \ref{T5.1}, we have the following two decay estimates: 
\begin{align}\label{Edecay}
\bega
&|E^P_{\al}(t)|  \leq \mathcal{E}(t)\leq C\mathcal{E}(0) e^{-\underline{a}t} \leq C\varepsilon e^{-\underline{a}t},\cr 
&|T_{\al}^P(t)-T_0|  \leq \mathcal{E}(t)+\frac{1}{2}|\bu^P(t)|^2 \leq C\varepsilon e^{-\underline{a}t} + C\varepsilon^2 e^{-2\underline{a}t} \leq C\varepsilon e^{-\underline{a}t},
\enda
\end{align}
where we used \eqref{vsmall}. For the last line of \eqref{mildEE}, we use \eqref{expansion} and apply \eqref{Edecay} to obtain
\begin{align}
\begin{aligned}\label{ETdecay}
\bigg|E_{\al}^P(s)+\frac{T_0^2}{T_{\al}^P(s)}-T_0\bigg| %&\leq |E_{\al}^P(s)|+T_0\bigg|\frac{T_0}{T_{\al}^P(s)}-1\bigg| \cr 
%&\leq C\varepsilon e^{-\underline{a}s} + T_0 \sum_{k\geq 1}\lw(-\frac{T_{\al}^P(s)-T_0}{T_0}\rw)^k \cr 
&\leq C\e e^{-\underline{a}s} + T_0 \sum_{k\geq 1}\lw(\frac{C\e e^{-\underline{a}s}}{T_0}\rw)^k.
\end{aligned}
\end{align}
Substituting \eqref{Edecay} and \eqref{ETdecay} on \eqref{mildEE} yields 
\begin{align}\label{mildEE2}
\begin{split}
&|E_{\al}^P(t) -E_{\al}^K(t)|  \\
& \hspace{0.5cm} \leq  e^{-\tilde{a}_{\al}t}|E_{\al}^P(0)-E_{\al}^K(0)| \cr 
& \hspace{0.5cm} +\int_0^t e^{-\tilde{a}_{\al}(t-s)} \bigg[2\tilde{a}_{\al} \left(C\e e^{-\underline{a}s} + T_0 \sum_{k\geq 1}\lw(\frac{C\e e^{-\underline{a}s}}{T_0}\rw)^k \right) \bigg] ds.
\end{split}
\end{align}
For $\e$ satisfying $C\varepsilon <T_0$, we have 
\[ T_0 \sum_{k\geq 1}\lw(\frac{C\e e^{-\underline{a}s}}{T_0}\rw)^k\leq 2C\e e^{-\underline{a}s}. \]
This yields
\begin{align*}
\begin{aligned}
|E_{\al}^P(t)-E_{\al}^K(t)| &\leq e^{-\tilde{a}_{\al}t}|E_{\al}^P(0)-E_{\al}^K(0)| + C\varepsilon \int_0^t e^{-\tilde{a}_{\al}(t-s)} \tilde{a}_{\al} e^{-\underline{a}s} ds.
\end{aligned}
\end{align*}
Finally, we apply \eqref{ee} on the time integration to obtain the desired estimate. 
\end{proof}

Now, we are ready to provide the proof of Theorem \ref{smdiff}. \newline

\noindent {\bf Proof of Theorem \ref{smdiff}} We combine Lemma \ref{u-ulem}, Lemma \ref{x-xlem} and Lemma \ref{E-Elem} and use $\underline{a}te^{-\underline{a}t} \leq Ce^{-\frac{1}{2}\underline{a}t}$ to have the desired result.

%%%%%%%%%%%%%%%%%%%%%%%%%%%%%%%%%%%%%%%%%%%%%%%%%%%%%%%%%%%%%%%%%%%%%
\section{Numerical Simulations} \label{sec:6}
\setcounter{equation}{0} 
In this section, we compare the dynamics of two particle models described in Section \ref{sec:5.2} using numerical simulations. Since solutions to the both models tend to the same asymptotic velocities, for the comparison, we focus on the dynamics in initial layer. 

%\subsection{Comparision of initial behaviors}\label{sec:6.2}

\subsection{Simulation set-up}  \label{sec:6.1}
We compare the PB-CS model \eqref{TCS} and the KB-CS model \eqref{KCS} according to the two different types of interaction weight. In both cases, we consider  one-dimensional case $d=1$. \\

\noindent$\bullet$~\textbf{Case A}: In this case, we follow the setting as in Section \ref{sec:5.2.1} which we recall here for convenience. Consider the uniform constant interaction weight as \eqref{a1}:
\begin{equation}\label{E6.2.1}
a_{\al\beta}=1 \quad \forall~\alpha, \beta \in [n].
\end{equation}
Here, we consider a three-particle system on $\mathbb{R}$.
\begin{align}
\begin{aligned}\label{E6.2.2}
\bu_{1}(0) &= 1, \quad \bu_{2}(0) = 2, \quad \bu_{3}(0) = -3, \cr 
T_{1}(0) &= 3, \quad T_{2}(0) = 0.01, \quad T_{3}(0) = 3.
\end{aligned}
\end{align}
Note that the dynamics with initial data \eqref{E6.2.2} is similar to \eqref{New-27}. Furthermore, we have
\[
T_0 \approx 4.3366.
\] 
Finally, initial positions are chosen randomly in $[-1/2,1/2]$ and then we rescale them to satisfy \eqref{inu}:
\[
x_1(0) \approx 0.2108, \quad x_2(0) \approx -0.3500, \quad x_3(0) \approx 0.1392.
\]
However, we stress that the choice of initial position does not affect the dynamics of $\bu_{\al}$ and $T_{\al}$. \\ 

\noindent$\bullet$~\textbf{Case B}:~The second case deals with the all-to-all symmetric interaction weight as in Section \ref{sec:5.2.2}. In particular, the interaction matrix $(a_{\alpha \beta})$ is given by
\begin{equation}\label{E6.2.3}
(a_{\al\beta})= \lw[\begin{array}{cccc} 0 & 100 & 1 & 1\\ 
100 & 0 & 1 & 1\\ 1 & 1 & 0 & 100 \\ 1 & 1 & 100 & 0\end{array} \rw].
\end{equation} 
Now we fix the initial position and temperature by 
\begin{align*}
&x_1(0) \approx 0.3709, \quad x_2(0) \approx  -0.1899, \quad x_3(0) \approx 0.2992, \quad x_4(0) \approx -0.4802,\\
&T_{1}(0) = 1, \quad T_{2}(0) = 0.1, \quad T_{3}(0) = 1, \quad T_4(0) = 1.
\end{align*}
Again, we note that the choice of initial position does not influence the dynamics of $\bu_{\al}$ and $T_{\al}$. In this case, we consider two types of initial velocity for a four-particle system. The first set of initial data is given by
\begin{align}
\begin{aligned}\label{E6.2.4}
\bu_{1}(0) &= 2, \quad \bu_{2}(0) = 1.1, \quad \bu_{3}(0) = -1.1, \quad \bu_{4}(0)= -2.
\end{aligned}
\end{align}
Then, we implement the simulation with the initial data:
\begin{align}
\begin{aligned}\label{E6.2.5}
\bu_{1}(0) &= 1, \quad \bu_{2}(0) = 2, \quad \bu_{3}(0) = -1, \quad \bu_{4}(0)= -2.
\end{aligned}
\end{align}

\subsection{Simulation results} \label{sec:6.2}
We present the results of the simulations. \\

\noindent$\bullet$~\textbf{Case A}: In Figure \ref{Case1-1}, we can see the evolution of velocities of the PB-CS model (Figure \ref{Case1-1}-(a)) and the KB-CS model (Figure \ref{Case1-1}-(b)). In particular, Figurre \ref{Case1-1}-(b) shows the explicit solution of the KB-CS model obtained in \eqref{New-26}. The blue lines in Figure \ref{Case1-1}-(a) and (b) represent $\bu_1$ in each of the models. The velocity $|\bu_1|$ in the PB-CS model initially increases as \eqref{New-28} shows. In Figure \ref{Case1-1}-(c), we present $\bu_1$ of the PB-CS model (red line) and the KB-CS (blue line) with different scales. We can clearly see the increase of $|\bu_1|$ in the PB-CS model which is consistent with \eqref{New-28}.

\begin{figure}[h]
\centering 
\subfigure[PB-CS velocity]{
\includegraphics[width=0.3\linewidth]{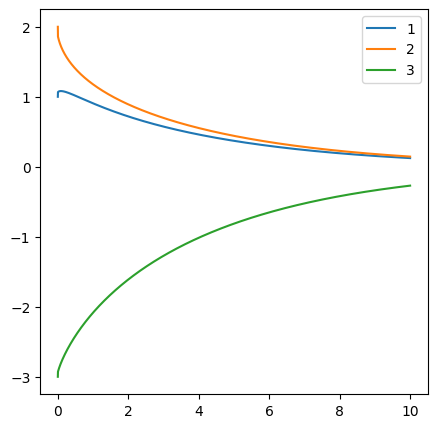} }
\subfigure[KB-CS velocity]{
\includegraphics[width=0.3\linewidth]{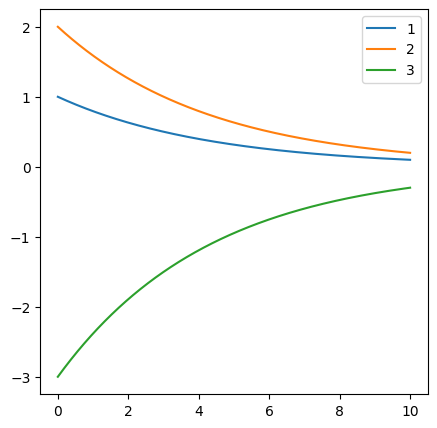} }
\subfigure[Velocities of $\bu_1$ in PB-CS and KB-CS]{
\includegraphics[width=0.3\linewidth]{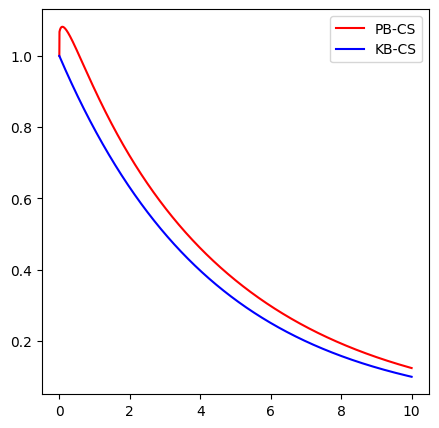} }
\caption{The dynamics of velocities with time-step $=0.001$. (c) is the zoomed-in picture of $\bu_1$ in PB-CS and KB-CS models. } 
\label{Case1-1}
\end{figure}
In Figure \ref{Case1-2}, we describe the evolution of temperatures in the PB-CS model (Figure \ref{Case1-2}-(a)) and the KB-CS model (Figure \ref{Case1-2}-(b)). We observe that in Figure \ref{Case1-2}-(a), the first temperature $T_1$ (blue line) initially decreases even if it starts below $T_0$. Indeed, in the same way as we compute in \eqref{New-28}, we have
\[
\frac{dT_1}{dt}\bigg|_{t=0} = -\frac{595}{9}T_0 - \frac{299}{9}T_0^2 < 0,
\]
which explains the initial behavior of $T_1$ in the PB-CS model. 
\begin{figure}[h]
\centering 
\subfigure[PB-CS temperature]{
\includegraphics[width=0.3\linewidth]{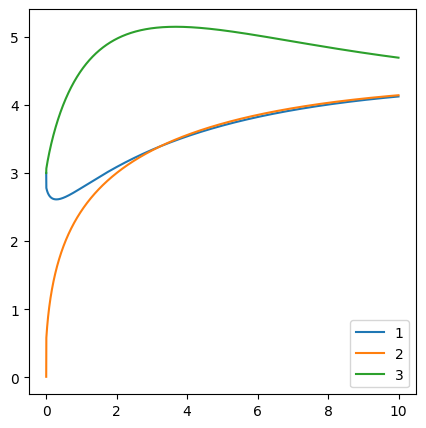} }
\subfigure[KB-CS temperature]{
\includegraphics[width=0.3\linewidth]{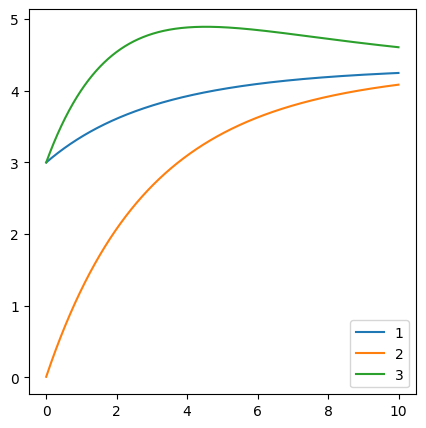} }
\caption{The dynamics of temperatures with time-step $=0.001$. (a) illustrates the evolution of temperature in the PB-CS model and (b) describes the temperature in the KB-CS model.} 
\label{Case1-2}
\end{figure}
Finally, in Figure \ref{Case1-3}, we illustrate the evolution of the functionals $\mathcal{V}$ and $\mathcal{E}$ \eqref{XVdef}. The monotone decrease of $\mathcal{V}$ and $\mathcal{E}$ corresponds to \eqref{New-26.5} and \eqref{VTR}.
\begin{figure}[h]
\centering 
\subfigure[The functional $\mathcal{V}^2$]{
\includegraphics[width=0.3\linewidth]{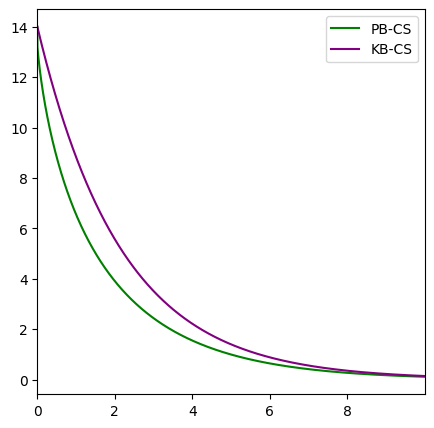} }
\subfigure[The functional $\mathcal{E}^2$]{
\includegraphics[width=0.3\linewidth]{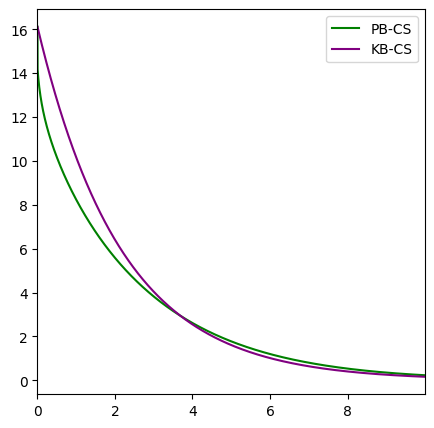} }
\caption{The time evolution of functionals $\mathcal{V}^2$ and $\mathcal{E}^2$ with time-step $=0.001$. The $y$-axis of (a) represents the value $\mathcal{V}^2(t)$ at time $t$ while the $y$-axis of (b) represents $\mathcal{E}^2(t)$.} 
\label{Case1-3}
\end{figure}

\newpage
\noindent$\bullet$~\textbf{Case B}: In Figure \ref{Case2u}, we can see the evolution of velocity $\bu_1(t)$ in the PB-CS model and the KB-CS model for the initial velocity \eqref{E6.2.5}.
This shows that the velocity $|\bu(t)|$ of the KB-CS model \eqref{KCS} can increase initially for a specific choice of the interaction matrix \eqref{E6.2.3}. This behavior is different from the case of uniform constant interaction weight \eqref{E6.2.1} as expected in \eqref{KCSinc}.

\begin{figure}[h]
\centering 
%\subfigure[Velocity $u_1$, $500$-iteration]{
%\includegraphics[width=0.4\linewidth]{5th.png} }
\subfigure[Velocity $u_1$, $10000$-iteration]{
\includegraphics[width=0.4\linewidth]{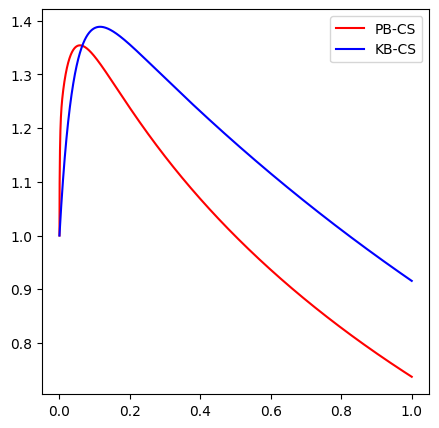} }
\caption{The time evolution of $\bu_1$ for the initial velocity \eqref{E6.2.5}.} 
\label{Case2u}
\end{figure}

%We implement numerical simulations which illustrate Proposition \ref{P5.1}, \ref{P5.2} and \ref{P5.3} as well as the initial increase of velocity as .

In Figure \ref{Case2a} and \ref{Case2b}, we can see the evolution of the functional $\mathcal{V}^2$ and $\mathcal{E}^2$ of the PB-CS model and KB-CS model, for the two initial data \eqref{E6.2.4} and \eqref{E6.2.5}, respectively. For the KB-CS model, we can see the monotone decrease of $\mathcal{V}_{KBCS}^2$ and $\mathcal{E}_{KBCS}^2$ proved in Proposition \ref{P5.1}.
On the other hand, for the PB-CS model, the functional $\mathcal{V}^2_{PBCS}$ can initially increase for the specific initial data \eqref{E6.2.4} as Proposition \ref{P5.2} shows. 
\begin{figure}[h]
\centering 
\subfigure[The functional $\mathcal{V}^2$, $100$-iteration]{
\includegraphics[width=0.4\linewidth]{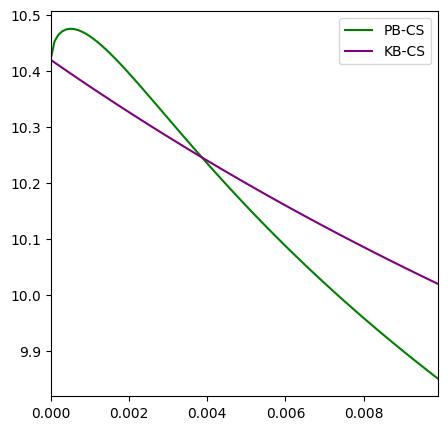} }
\subfigure[The functional $\mathcal{E}^2$, $100$-iteration]{
\includegraphics[width=0.4\linewidth]{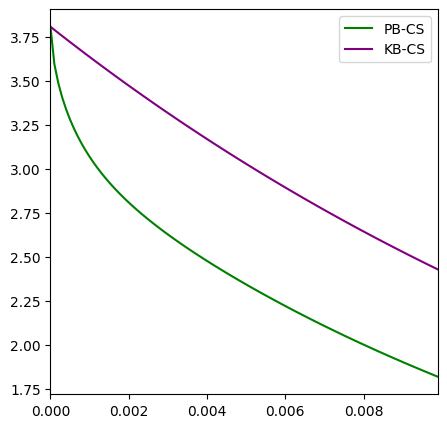} }
\caption{The time evolution of functionals $\mathcal{V}^2$, $\mathcal{E}^2$ having initial velocity \eqref{E6.2.4} with stepsize $=0.0001$.} 
\label{Case2a}
\end{figure}
For the initial data \eqref{E6.2.5}, we can see the initial increase of the functional $\mathcal{E}^2_{PBCS}$ in Figure \ref{Case2b} which describes Proposition \ref{P5.3}.
\begin{figure}[h]
\centering 
\subfigure[The functional $\mathcal{V}^2$, $500$-iteration]{
\includegraphics[width=0.4\linewidth]{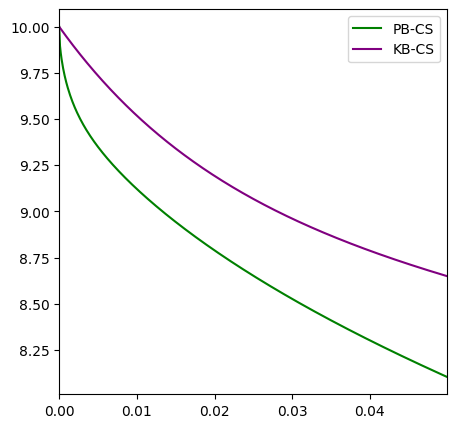} }
\subfigure[The functional $\mathcal{E}^2$, $500$-iteration]{
\includegraphics[width=0.4\linewidth]{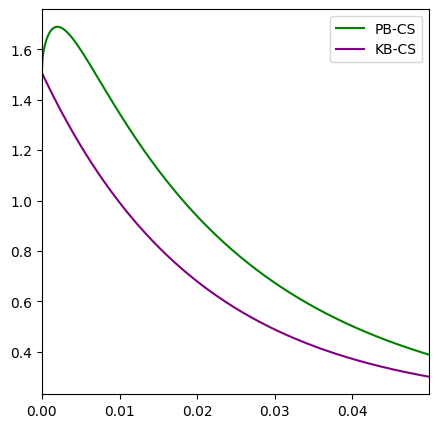} }
\caption{The time evolution of functionals $\mathcal{V}^2$, $\mathcal{E}^2$ having initial velocity \eqref{E6.2.5} with stepsize $=0.0001$.} 
\label{Case2b}
\end{figure}

\section{Conclusion} \label{sec:7}
\setcounter{equation}{0}
In this paper, we have studied production terms arising from the reduction of balance laws based on two theories, namely phenomenological macroscopic theory and kinetic theory for gas mixtures. In literature, the former has been known to satisfy an entropy principle, whereas it is not clear whether the latter satisfies the entropy principle or not. In this work, we showed that the production terms satisfy an entropy principle. We also adopt the reduction procedure employed for the derivation of the thermodynamic Cucker-Smale model to derive a new particle flocking model from the balance laws based on the kinetic theory for the mixture. We show that the kinetic theory-based particle model exhibits asymptotic flocking dynamics for all initial data without any restrictions on initial data. When initial data are close to some equilibrium state, we show that both models  satisfy asymptotic equivalence in velocity and energy, of course spatial positions can be made sufficiently small. There are several untouched issues for the new kinetic theory-based Cucker-Smale model. Kinetic and hydrodynamic descriptions for this new model have not been studied in literature. So it might be interesting problems to investigate the aforementioned problems in a future work. 

\section*{Acknowledgment}
The work of G.-C. Bae is supported by the National Research Foundation of Korea(NRF) grant funded by the Korea government(MSIT) (No. 2021R1C1C2094843), and 
the work of S.-Y. Ha was supported by National Research Foundation of Korea(NRF-2020R1A2C3A01003881). The work of T.~Ruggeri was carried out in the framework of the activities of the Italian National Group for Mathematical Physics [Gruppo Nazionale per la Fisica Matematica (GNFM/INdAM)].

\end{document}